\documentclass[smallextended]{svjour3}     
\smartqed  
\usepackage{graphicx}

\usepackage{savesym}
\usepackage{amsmath}
\usepackage{amssymb}
\usepackage{txfonts}      

\RequirePackage[numbers]{natbib}
\usepackage{url}
\newcommand{\href}[2]{{#2}}

\RequirePackage{ifthen}

\newlength{\ei}\ei=0.0138888889em
\newlength{\eN}
\newlength{\SyW}  %
\newlength{\msu}  \msu=\mathsurround 
    
 \newcommand{\Ts}{\textstyle}
\newcommand{\Ss}{\scriptstyle}  \newcommand{\SSs}{\scriptscriptstyle}
\newcommand{\ssr}{\rm\scriptscriptstyle}
\newcommand{\req}{\relax}
\newcommand{\rfia}[1]{\makebox[\parindent][l]{%
                     \makebox[0em][r]{\rm(}\sf#1\rm)}}
\newcounter{ABCcB}
\newcommand{\theABCcC}{\alph{ABCcB}}

\newcommand{\Ew}{\mathop{\rm {{}E{}}}\nolimits} 
\newcommand{\Var}{\mathop{\rm Var}\nolimits}    
\newcommand{\sign}{\mathop{\rm sign}\nolimits}    
\newcommand{\ve}{\varepsilon}

\newcommand{\Lw}{{\cal L}}
\newcommand{\Lo}{\mathop{\rm {{}o{}}}\nolimits}
\newcommand{\LO}{\mathop{\rm {{}O{}}}\nolimits}

\newcommand{\B}{\mathbb B}

\newcommand{\R}{\mathbb R}

\newcommand{\N}{\mathbb N}
\newcommand{\Jc}{\mathop{\bf\rm{{}I{}}}\nolimits}

\newcommand{\iid}{\mathrel{\stackrel{\ssr i.i.d.}{\sim}}}

\newcommand{\Tfrac}[2]{\textstyle\frac{#1}{#2}}

\newcommand{\Tsum}{\mathop{\Ts\sum}\nolimits}

\newtheorem{Thm}{Theorem}
\newtheorem{Prop}[Thm]{Proposition}
\newtheorem{Lem}[Thm]{Lemma}
\newtheorem{Rem}[Thm]{Remark}
\newtheorem{Cor}[Thm]{Corollary}
\newtheorem{Def}[Thm]{Definition}

\newtheorem{Bez}[Thm]{Notation}

\numberwithin{equation}{section}
\numberwithin{Thm}{section}

\newcounter{ABCc}
\renewcommand{\theABCc}{\alph{ABCc}}
\newenvironment{ABC}{\begin{list}{
  \rfia{\theABCc}}{\usecounter{ABCc} \topsep 0ex \partopsep 0ex \itemsep0ex
  \parsep=\parskip \leftmargin 0em \rightmargin 0em \itemindent=\parindent
  \listparindent=\parindent  \labelsep 0.2em \labelwidth 0.5em }}{\end{list}}

\newcommand{\pkg}[1]{\texttt{#1}}

\newcommand{\asy}{\mbox{as\hskip-4\ei .\hskip-16\ei.}\hskip12\ei}

\ifx\blinded\undefined
\title{Higher Order Expansion for the MSE of M-estimators on shrinking neighborhoods}

\author{Peter Ruckdeschel}

\institute{P. Ruckdeschel \at
              Fraunhofer ITWM, Department of Financial Mathematics, \\
              Fraunhofer-Platz 1, 67663 Kaiserslautern, Germany\\
              and Dept.\ of Mathematics, University of Kaiserslautern,\\
              P.O.Box 3049, 67653 Kaiserslautern, Germany \\
              \email{peter.ruckdeschel@itwm.fraunhofer.de}\\           
}
\else
\title{Higher order Expansion for the MSE of M-estimators on shrinking neighborhoods}
\author{}\institute{}
\fi
\date{Received: date / Accepted: date}
\begin{document}
\maketitle
\begin{abstract}
We consider estimation of a one-dim.\ location parameter
by means of M-estimators $S_n$ with monotone influence curve $\psi$.
For growing sample size $n$,  on suitably thinned out convex contamination balls
$\tilde{\cal Q}_n$ of shrinking radius $\textstyle{{r}/{\sqrt{n}}}$
about the ideal distribution, we obtain an expansion of the asymptotic
maximal mean squared error MSE of form
$$
\max_{Q_n\in \tilde{\cal Q}_n}
n \,{\rm MSE}(S_n,Q_n) = r^2 \sup \psi^2 +
{\rm E}_{\rm \SSs id}\psi^2+ {\Tfrac{r}{\sqrt{n}}}\,
A_1 + {\Tfrac{1}{n}}\,A_2+\Lo({\Tfrac{1}{n}}),
$$
where $A_1$, $A_2$ are constants depending on $\psi$ and $r$.
Hence $S_n$ not only
is uniformly (square) integrable in $n$ (in the ideal model) but also
on $\tilde{\cal Q}_n$, which is not self-evident.
For this result, the thinning of the neighborhoods, by a breakdown-driven, sample-wise 
restriction, 
is crucial, but exponentially negligible.
Moreover, our results essentially characterize contaminations
generating maximal ${\rm MSE}$ up to $\Lo(n^{-1})$. 
Our results are confirmed empirically by simulations as well as numerical evaluations
of the risk.
\keywords{higher order asymptotics \and location M-estimator \and uniform integrability
\and Edgeworth expansion \and gross error neighborhood \and shrinking neighborhood
\and breakdown point
}%
\subclass{MSC 62F35, 62F12}
\end{abstract}

\section{Motivation/introduction}\req
In the setup of shrinking neighborhoods about a general, parametric
 ideal central model, \citet{Ri:94} determines the 
asymptotically linear estimator (ALE) minimaxing \asy ${\rm MSE}$ 
on these neighborhoods.
We address the question to which degree this asymptotic optimality
carries over to finite sample size and try to identify
and quantify which aspects of both estimator and neighborhood
are responsible for the quality of the approximation.

\subsection{Setup: one-dimensional location} \label{1dimlocset}
As a starting point for assessing such questions we consider the most
basic parametric model of statistics, the one-dimensional location model
$\{P_{\theta}(dx)=F(dx-\theta),\;\theta\in\R\}$
for some ideal distribution $F$ with finite Fisher-Information of location ${\cal I}(F)$
in the sense of \citet[4.Def.4.1, Thm.4.2]{Hu:81},  i.e.\
${\cal I}(F):=\sup_{\varphi\in{\cal C}_c^1} (\int \dot \varphi \,dF)^2/(\int \varphi^2 \,dF)$, entailing that
$\Lambda_f=-\dot f/f \in L_2(F)$, $ {\cal I}(F)=\Ew[\Lambda_f^2]$.
Paralleling \citet{Hu:64}, we also assume that $\Lambda_f$ is increasing.
By translation equivariance, we may restrict ourselves to $\theta_0=0$ which
is suppressed in the notation.

The set of {\em influence curves\/} (IC's) $\Psi$ in this model is defined as in
\citet{Ri:94}
\begin{equation}
  \Psi:=\{\psi \in L_2(F) \,|\, \Ew[\psi]=0,\quad \Ew[\psi \Lambda_f]=1\},
\end{equation}
where both expectations are evaluated under $F$.
\paragraph{Shrinking neighborhoods} \label{110ssn}
Robust Statistics enlarges the ideal model assumptions by suitable neighborhoods
about them. The shrinking neighborhood approach---compare e.g.\ \citet{Ri:94},
\citet{K:R:R:10}, balances bias and variance, which would be of different scaling
in $n$ otherwise, see also \citet{Ruck:03e}. For this paper
we consider contamination neighborhoods, i.e.\ the set ${\cal Q}_n(r)$ of distributions
\begin{equation}\label{contadef}
{\cal L}^{\SSs \rm real}_{\theta}(X_1,\ldots,X_n)=Q_{n}=
{\Ts\bigotimes\limits_{i=1}^n}[(1-\Tfrac{r_n}{\sqrt{n}}) F+\Tfrac{r_n}{\sqrt{n}}\,P^{\SSs \rm di}_{n,i}]
\end{equation}
with $r_n=\min(r,\sqrt{n})$, $r>0$ the contamination radius and $P^{\SSs \rm di}_{n,i} \in {\cal M}_1(\B)$ arbitrary,
uncontrollable contaminating distributions.
As usual, we interpret $Q_n$ as the distribution
of the vector $(X_i)_{i\leq n}$ with components
\begin{equation}\label{Uidef}
  X_i:=(1-U_i)X_i^{\SSs \rm id}+U_i X_i^{\SSs \rm di},\qquad i=1,\ldots,n
\end{equation}
for $X_i^{\SSs \rm id}$, $U_i$, $X_i^{\SSs \rm di}$ stochastically independent, $X_i^{\SSs \rm id}\iid F$, $U_i\iid{\rm Bin}(1,r/\sqrt{n})$,
and $(X_i^{\SSs \rm di})\sim P^{\SSs \rm di}_{n}$ for some arbitrary $P^{\SSs \rm di}_{n} \in{\cal M}_1(\B^n)$.
\paragraph{First order optimality} \label{fooptgen}
For a sequence of estimators $S_n$, consider as risk the asymptotically (modified) maximal MSE on ${\cal Q}_n$
\begin{equation} \label{modifmse}
\tilde R(S_n,r):= \lim_{t\to\infty}\lim_{n\to\infty}\sup_{Q_n\in{\cal Q}_n(r)} \int \min\{t,\, n\,|S_n-\theta_0|^2\} \,dQ_n
\end{equation}
Following \citet[Ch.~5]{Ri:94} a (suitably constructed) ALE $S_n$ with IC $\psi$ has risk
\begin{equation}\label{MSEpsi}
\tilde R(S_n,r)=r^2\sup |\psi|^2 + \Ew_{\rm \SSs id} |\psi|^2
\end{equation}
By 
Theorem~5.5.7 (ibid.), together with its preceding remarks, for given $r\ge 0$,
a (suitably constructed) ALE with IC $\hat \eta$ minimizes $\tilde R(\,\cdot\,,r)$ among all ALE's
iff $\hat \eta=\eta_{c_0}$
 for Lagrange multipliers $z$ and $A$ such that $\eta_{c_0}$ is an IC for
\begin{eqnarray}\label{HK1}
&&  \eta_{c_0}=A(\Lambda_f-z)\min\{1,
c_0/{|\Lambda_{f}-z|}\},\qquad
\Ew[(|\Lambda_f-z|-c_0)_+]=r^2c_0 \label{HK3}
\end{eqnarray}
\paragraph{Open issues in this setup}~\label{short}
Being bound to first order asymptotics, so far these results do not come along with an
indication for the speed of the convergence; it is not clear to what degree
radius $r$, sample size $n$ and clipping height $b$ affect this approximation.
The theorem only characterizes the optimal expansion in terms of ICs.
Finally, modification~\eqref{modifmse} of the MSE, which is common in asymptotic  statistics,
cf.\  \citet{LC:86}, \citet{Ri:94}, \citet{BKRW:98}, \citet{VdW:98},  and which forces the integrals to
converge under weak convergence, appears somewhat ad hoc. One would perhaps prefer a modification
that is statistically motivated.
\subsection{M-estimators for location}
There are several constructions for an ALE to achieve a given IC $\psi$---one-step constructions,
M-estimators, L-estimators and many more. In this paper we confine ourselves 
to M-estimators. We require $\psi$ to be monotone and bounded and write 
$\psi_t(\,\cdot\,)$ for $\psi(\,\cdot\,-t)$.
For technical reasons we assume that the  set $D_t$ of
discontinuities of the c.d.f.\ of $\psi_t(X^{\rm \SSs id})$
has to carry less mass than $1$ uniformly:
\begin{equation} \label{pdt}
p_{D}:=\sup\nolimits_t\, P^{\rm \SSs id}(D_t)<1
\end{equation}
Following the notation in \citet[pp.~46]{Hu:81},  let
\begin{equation}\label{mest1}
  S_n^{\ast}:=\sup\Big\{\,t\,|\,\sum_{i\leq n} \psi_t(x_i)>0\Big\},\qquad
  S_n^{\ast\ast}:=\inf\Big\{\,t\,|\,\sum_{i\leq n} \psi_t(x_i)<0\Big\}
\end{equation}
and $S_n$ be any estimator satisfying $S_n^{\ast}\leq S_n \leq S_n^{\ast\ast}$.
By monotonicity of $\psi$, we get
\begin{equation}\label{mest2}
  \Pr\{S_n^{\ast}<t\}=\Pr\Big\{\sum_{i\leq n} \psi_t(x_i)\leq 0\Big\},\qquad
  \Pr\{S_n^{\ast\ast}<t\}=\Pr\Big\{\sum_{i\leq n} \psi_t(x_i)<0\Big\}
\end{equation}
in the continuity points $t$ of the LHS.
The next lemma, an immediate consequence of \citet[Theorem~2.3]{Hal:92}, shows that we may ignore the event
$S_n^{\ast}\not=S_n^{\ast\ast}$ if we are
interested in statements valid up to
$\Lo(1/n)$.
\begin{Lem}\label{Kohllem}
Under \eqref{pdt}, $\Pr(S_n^{\ast}\not=S_n^{\ast\ast})=\LO(\exp(-\gamma n))$ for some $\gamma>0$.
\end{Lem}
\begin{Rem}\rm\small
If $\bigcup_t D_t=\{\pm c\}$ for some $c>0$, $\Pr(S_n^{\ast}\not=S_n^{\ast\ast})=0$ for $n$ odd.
\end{Rem}%
\subsection{Organization of this paper and description of the results}
This paper provides answers to some of the open questions mentioned in subsection~\ref{short}; these
answers were initiated by an attempt to check the validity of Rieder's asymptotic approach at finite sample
sizes by simulations in 2003. 
At closer inspection of these simulations,  
M.~Kohl found out that larger inaccuracies of (first order) asymptotics  only occurred in extraneous sample situations
where more than half the sample size stemmed from a contamination, which made him conjecture that excluding such samples,
 asymptotics might then prove useful even for very small samples.
With regard to our shrinking setup, such an exclusion on the one hand
is asymptotically negligible, hence does not affect the results of
subsection~\ref{110ssn}, but on the other hand under this restriction indeed 
the unmodified MSE converges along with weak convergence.
We discuss this modification in section~\ref{modifsec}. 
In section~\ref{resultsec}, we present
the central theoretical result, Theorem~\ref{main}. This result is of the following form
\begin{equation}\boldmath\label{typresult}
  \sup_{Q_n\in\tilde{\cal Q}_n(r;\ve_0)}n\, {\rm MSE}(S_n,Q_n)=r^2 \sup |\psi|^2+ \Ew \psi^2 + \Tfrac{r}{\sqrt{n}\,}\,A_1+
  \Tfrac{1}{n}\,A_2+\Lo(\Tfrac{1}{n})
\end{equation}
Here $S_n$ is an M-estimator to IC $\psi$, and $A_1$, $A_2$ are polynomials in the contamination radius $r$,
in $b=\sup |\psi|$,
and in the moment functions
$t\mapsto \Ew \psi_t^l$, $l=1,\ldots,4$ and their derivatives evaluated in $t=0$.
We recognize at once that the speed of the convergence to the first order asymptotic  
value is one order faster in the ideal model.\vspace{-1ex}
\begin{Bez}
  \rm\small
For indices we start counting with $0$, so that terms of first-order asymptotics
  have an index $0$, second-order ones a $1$ and so on. Also we abbreviate first-order, second-order and third-order
  by f-o, s-o, t-o respectively, and we write f-o-o, s-o-o, and t-o-o for first, second, and third-order asymptotically  optimal respectively.\vspace{-1ex}
\end{Bez}
As to the correctness of our main result, we give a number of cross checks and comments on this result
in section~\ref{relsec}. The relevance of these results for (small) finite sample sizes is shown
by a simulation study which is presented in section~\ref{simsec} as to its design and results.
By means of an adopted convolution algorithm taken from \citet{R:K:10:FFT},  we also compute numerically
exact values of the MSE.
Proofs are delegated to the appendix  section~\ref{proofsec}.
These contain rather tedious Taylor expansions where we need the help of a symbolic Algebra program like {\tt MAPLE}.
To ease readability, we therefore start the proof of the main theorem with an outline of the essential steps.
Some auxiliary results needed in the proofs are provided in an appendix in section~\ref{appsec}.

On a web-page to this page, additional tables and figures,
the {\tt MAPLE} script to generate the expansions, and the {\tt R}-script to calculate numerically exact MSE
are available for download.
\section{Modification of the shrinking neighborhood setup}\label{modifsec}\req
The key property in the shrinking-neighborhood setup is the LAN-property\footnote{for local asymptotic  normality}
in the sense of H\'ajek and LeCam. LAN holds for $L_2$-differentiable models, c.f.~\citet[Thm.\ 2.3.5]{Ri:94}.
and together with LeCam's third Lemma---c.f.\ 
Cor.~2.2.6 ibid.
---implies uniform weak
convergence of any (suitably constructed) ALE to a bounded IC
on a representative subclass of the system of neighboring
distributions ${\cal Q}_n$---those distributions induced by
simple perturbations $Q_n(\zeta,t)$, see     
p.~126 (ibid.)
.\\
Without additional assumptions, this weak convergence however does not carry over
to convergence of the risk for an unbounded loss function in general, i.e.\
uniform integrability fails on any proper neighborhoods shrinking arbitrarily fast;
which can be seen along the lines of \citet[Prop.~2.1]{Ruck:03b}.
%
%
%
%
\paragraph{Modification of the shrinking neighborhood setup}
%
We instead propose the following modification of the neighborhoods for finite $n$:
Only realizations of $U_1,\ldots,U_{n}$ are permitted, where $\sum U_i < n/2$.
More precisely, accounting for non-symmetric $\psi$,
we introduce
\begin{equation}\label{checkbdef}
\check b:= \inf\psi,\qquad\hat b =\sup\psi,\qquad  \bar b:=\Tfrac{1}{2}(\hat b -\check b),\qquad
\delta_0:=\Tfrac{|\check b+\hat b|}{\min((-\check b),\hat b)}\ge 0
\end{equation}
and recall that in our situation, both the functional (\citet[(2.39),(2.40)]{Hu:81}) and the finite sample ($\ve$-contamination) breakdown
point (\citet[section 2.2]{Do:Hu:83}) of $T$ respectively $S_n$ are
\begin{equation}
\ve_0=1/({2+\delta_0})={\sup|\psi|}/({\hat b - \check b})
\end{equation}
With these expressions, our modification amounts to considering the neighborhood system $\tilde {\cal Q}_n(r;\ve_0)$
of conditional distributions
\begin{equation}\label{modifdefi}
Q_{n}={\cal L}\Big\{[(1-U_i)X_i^{\SSs \rm id}+U_i X_i^{\SSs \rm di}]_i\,\Big|\,\sum U_i \leq\,\ulcorner \ve_0 n\, \urcorner-1\,\Big\}
\end{equation}
This restriction hence combines a restriction to the marginals $\Lw(X_i^{\SSs \rm real})$ which are ``close''
to $\Lw(X_i^{\SSs \rm id})$ for each $i$ as well as a sample-wise restriction.\\
Correspondingly, we will consider the asymptotics of the \textit{unmodified} MSE risk
\begin{equation} \label{origmse}
R_n(S_n,r;\ve_0):= \sup_{Q_n\in \tilde{\cal Q}_n(r;\ve_0)} n\, \int |S_n-\theta_0|^2 \,dQ_n
\end{equation}

\paragraph{Asymptotic negligibility of this modification}\label{expneg}
The effect of this modification is negligible asymptotically:
By the Hoeffding bound \eqref{hoe2},
\begin{eqnarray}
P(\Tsum U_i\ge n\ve_0)
&\leq&\exp\Big(-2n (\ve_0-{r}/{\sqrt{n\,}}\,)^2 \Big) \label{expneg2}
\end{eqnarray}
which decays exponentially fast. Thus all results on convergence in law of the shrinking neighborhood
setup are not affected when passing
from ${\cal Q}_n(r)$ to $\tilde {\cal Q}_n(r;\ve_0)$.
%
%
\begin{Rem}\rm\small
\begin{ABC}
  \item Thinning out the neighborhoods is equally relevant for the interchange of integration and maximization
in the context of neighborhoods to a fixed radius $\ve$:
  Replacing $r/\sqrt{n}$ by  $\ve$, asymptotic  negligibility \eqref{expneg2} continues to hold, as long as $\ve<\ve_0$,
  while the failure of uniform integrability persists.
\item M-estimators have the well-known feature that in general 
the procedure with optimal efficiency [minimax MSE in our context]
does not attain maximal breakdown point [works with minimally thinned out
neighborhoods]; but just as already mentioned in \citet{RouP:84} and similarly as worked out in \citet{Yo:87},  
both goals may be achieved simultaneously combining a starting M-estimator of maximal breakdown point with a correction by a one-/$k$-step
construction with the f-o-o (or s-o-o, t-o-o) IC. 
\end{ABC}
\end{Rem}
\section{Main Theorem}\label{resultsec}\req
%
\paragraph{Notation}
To $\psi:\R\to\R$ monotone let $\psi_t(x):=\psi(x-t)$ and $\psi_t^0:=\psi_t-\Ew \psi_t$ define the following
functions
\begin{equation}
 L(t)=\Ew \psi_t,\;\; V(t)^2=\Ew (\psi_t^0)^2, \;\;
 \rho(t)=\Ew(\psi_t^0)^3 V(t)^{-3},\;\; \kappa(t)=\Ew(\psi_t^0)^4V(t)^{-4}-3\label{LVRK2}
\end{equation}
Let $\check y_n$ and  $\hat y_n$ 
sequences in $\R$ such that for some $\gamma>1$
\begin{equation}\label{subyndef}
  \psi(\check y_n)=\inf \psi +\Lo(\Tfrac{1}{n^\gamma}),\qquad   \psi(\hat y_n)=\sup \psi +\Lo(\Tfrac{1}{n^\gamma}) 
\end{equation}
For $H\in{\cal M}_1(\B^n)$ and an ordered set of indices $I=(1\leq i_1<\ldots<i_k\leq n)$ denote $H_{I}$ the marginal of
$H$ with respect to $I$.
\begin{Def}\label{thinout}
Consider sequences $c_n$, $d_n$, and   $\kappa_n$ in $\R$, in $(0,\infty)$, and in $\{1,\ldots,n\}$, respectively.
We say that $(H^{(n)})\subset{\cal M}_1(\B^n)$ is {\em $\kappa_n$--concentrated left [right] of $c_n$ up to $\Lo(d_n)$}, if
for each sequence of ordered sets $I_n$ of cardinality $i_n\leq \kappa_n$
\begin{equation}
1-H^{(n)}_{I_n}\big((-\infty;c_n]^{i_n}\big)=\Lo(d_n)
\qquad
\Big[\,
1-H^{(n)}_{I_n}\big((c_n,\infty)^{i_n}\big)=\Lo(d_n)\,
\Big]
\end{equation}
\end{Def}
\paragraph{General assumptions in this paper}\label{Assum}
\begin{itemize}
  \item[(bmi)]  $\sup\|\psi\| =b <\infty$, $\psi$ monotone, $\psi \in \Psi$
  \item[(D)]   For some $\delta\in(0,1]$, $L$, $V$, $\rho$,   and  $\kappa$ from \eqref{LVRK2} allow the expansions
\begin{align}
&  \hspace{-2em}L(t)=l_1t+\Tfrac{1}{2} l_2\,t^2+\Tfrac{1}{6}l_3\,t^3+\LO(t^{3+\delta}) ,
   &&V(t)=v_0(1+\tilde v_1\,t+ \Tfrac{1}{2}\tilde v_2\,t^2)+\LO(t^{2+\delta})\label{VD}\hspace{1em}\mbox{ }\\
&\hspace{-2em}  \rho(t)=\rho_0+\rho_1\, t+\LO(t^{1+\delta}),
  &&\kappa(t)=\kappa_0+\LO(t^{\delta})\hspace{1em}\mbox{ }\label{KD}
\end{align}
  \item[(Vb)] $V(t)=\LO( |t|^{-(1+\delta)})$ for $|t|\to \infty$ and some $\delta\in(0,1]$
  \item[(C)] Let $f_t$ be
  the characteristic function of $\psi_t(X^{\rm\SSs id})$; then
  \begin{equation}\label{(C)gl}
\lim_{t_0\to 0}\limsup_{s\to\infty} \sup_{|t|\leq t_0}|f_t(s)|<1
  \end{equation}
  \end{itemize}
Condition (C) is a local uniform Cram\'er condition; it is implied by
\begin{Lem}\label{clemn}
Assume ${\cal L}(\psi(X^{\rm\SSs id}))$ has a nontrivial absolute continuous part and that $\psi$ is continuous.
 Then (C) is fulfilled.
\end{Lem}
%
%
\begin{Rem}\rm\small \label{varianteb}
\begin{ABC}
  \item By condition~(bmi) ---as $\psi \in \Psi$---, $l_1=-1$.
  \item   Condition (C) is not fulfilled for the median, as its influence curve just takes the values $-b,b$ $F$-a.e.
  A direct proof for an analogue to Theorem~\ref{main} is possible, however, and given in \citet{Ruck:03b}.
  \item For an expansion of the MSE up to  $\Lo(n^{-1/2})$, the $\kappa$ part of assumption~\eqref{KD} can be dropped, and we may use
  assumptions
  \begin{ABC}
  \item[(D')] For some $\delta\in(0,1]$, $L$, $V$, and $\rho$ allow the expansions
\begin{equation}
   L(t)=l_1t+l_2/2\,t^2+\LO(t^{2+\delta}), \qquad
   V(t)=v_0(1+\tilde v_1\,t)+ \LO(t^{1+\delta}), \qquad
  \rho(t)=\rho_0+\LO(t^{\delta}) \label{rho0B}
\end{equation}
  \item[(C')]  There exist $t_0>0$, $s_0>0$ such that for all $s_1>s_0$
  \begin{equation}\label{(C)gl2}
  \hat f_{s_0,t_0}(s_1):=\sup_{s_0\leq s\leq s_1} \sup_{|t|\leq t_0}|f_t(s)|<1
  \end{equation}
  \end{ABC}
\footnotesize{
Note that (C) implies (C'), but contrary to (C), in (C') the case $\sup_{s_1}\hat f_{s_0,t_0}(s_1)=1$ for all $s_0>0$ and all $t_0>0$ is allowed.}
\end{ABC}
\end{Rem}
\paragraph{Illustration}\label{illust}
%
We specialize the assumptions for $F={\cal N}(0,1)$, i.e. $\Lambda_f(x)=x$,
and $\psi(x)=\hat\eta_c(x)=A_c x\min\{1,{c}/{|x|}\}$ from (\ref{HK1}) with $A_c$ such that $\hat\eta_c\in\Psi$:
\begin{Prop}\label{normalspclem}
  For $F={\cal N}(0,1)$ and  for $\psi=\eta_c$ an influence curve to $c\in(0,\infty)$ of Hampel-form
    $\eta_c=A_c(x\min\{1,{c}/{|x|}\}$
with $A_c=(2\Phi(c)-1)^{-1}$,
   assumptions {(bmi)} to {(C)} are in force; in particular the bounds in {(Lb)} and {(Vb)}
   are even exponential.\\ 
    With $\Phi(x)$ the c.d.f.\ of ${\cal N}(0,1)$ and $\varphi(x)$ its density,
    we obtain
    $l_2=0$, $\tilde v_1=0$, $\rho_0=0$.
For $c\in(0,\infty)$, we get

\begin{small}
\begin{eqnarray*}
   l_3&=&  \frac{2c\varphi(c)}{(2\Phi(c)-1)},\quad
     v_0^2=2 b^2(1-\Phi(c))+A_c(1-2b\varphi(c)),\quad
\tilde     v_2=\frac{6\Phi(c)-4\Phi(c)^2-2-2c\varphi(c)}
{2 c^2(1-\Phi(c))+2\Phi(c)-1-2c\varphi(c)}\\
\rho_1&=&\frac{3 A_c^{3}\left(1-2\Phi(c)+2c\varphi(c)\right)}{v_0^3}+3v_0^{-1},\quad
\kappa_0=\frac{2c^4\left(1-\Phi(c)\right)-2c(c^2+3)\varphi(c)+3(2\Phi(c)-1)}{[2c^2\left(1-\Phi(c)\right)+2\Phi(c)-1-2c\varphi(c)]^2}-3
\end{eqnarray*}
\end{small}

For $c\downarrow 0$,
$l_3=1$, $v_0^2=\frac{\pi}{2}$, $\tilde v_2=- \Tfrac{2}{\pi}$, $\rho_1=2\sqrt{\Tfrac{2}{\pi}}$, $\kappa_0=-2$,
and formally, for $c\uparrow\infty$,
$l_3=0$, $v_0=1$, $\tilde v_2=0$, $\rho_1=0$, $\kappa_0=0$.
%
\end{Prop}
%
%
%
\begin{Thm}[Main Theorem]\label{main}
  In our one-dim.\ location model assume {(bmi)} to {(C)}  
\begin{ABC}
  \item the maximal MSE of the M-estimator $S_n$ to scores-function $\psi$
  expands to
  \begin{eqnarray}
    R_n(S_n,r,\ve_0)
    &=&{r}^{2}{b}^{2}+{v_0}^{2}+\Tfrac{r}{\sqrt{n}\,}\,A_1+\Tfrac{1}{n}\,A_2+\Lo(n^{-1})\label{mainres}
  \end{eqnarray}
with
\begin{eqnarray}
A_1&=&  {v_0}^{2}\,\Big( \pm(4\,\tilde v_1+3\,l_2 \,)b+1 \Big)+{b}^{2} +
 [2\,{b}^{2}\pm l_2\,{b}^{3} \,]\,{r}^{2}\label{A1defr}\\
A_2&=&{{v_0}^{3}\,\Big((l_2+2\,\tilde v_1 \,)\rho_0+\Tfrac{2}{3}\,\rho_1\Big)+
 {v_0}^{4}\,(3\,\tilde v_2+{\Tfrac {15}{4}}\,{l_2^2}+l_3+9\,{\tilde v_1}^{2}+
 12\,\tilde v_1\,l_2 \,)}+\nonumber\\
 &&\quad+[\, {v_0}^{2}\,\Big( (3\,\tilde v_2+3\,{\tilde v_1}^{2}+\Tfrac{15}{2}\,{l_2^2}+2\,l_3+
 12\,\tilde v_1\,l_2 \,){b}^{2}+1\pm (8\,\tilde v_1+6\,l_2 \,)\,b \Big)+\nonumber\\
 &&\quad\hphantom{+[}\pm 3\,l_2\,{b}^{3}+
 5 \,{b}^{2} \,]\,{r}^{2}+\Big( (\Tfrac{5}{4}\,{l_2^2}+\Tfrac{1}{3}\,l_3 \,){b}^{4}
 \pm 3\,l_2\,{b}^{3} +3\,{b}^{2}\Big)\,{r}^{4}\label{A2defr}
\end{eqnarray}
and we are in the $-\,[+]$-case depending on whether  \eqref{contbed1} or \eqref{contbed2} below applies.
\item
let $P^{\SSs \rm di}_{n}:=\bigotimes_{i=1}^n P^{\SSs \rm di}_{n,i}$ be contaminating measures for \eqref{contadef}.
Then $Q_n$ with $P^{\SSs \rm di}_{n}$ as  contaminating measures generates maximal risk in \eqref{mainres} if
for $k_1>1$ and $k_2>2\vee (\frac{3}{2}+\frac{3}{2\delta})$ with $\delta$ from (Vb)
and $K_1(n)=\ulcorner k_1 r\sqrt{n} \urcorner$
either
\begin{equation}
  \label{contbed1}
(P^{\SSs \rm di}_{n}) \mbox{ is  $K_1(n)$--concentrated left of $\check y_n- b\sqrt{k_2\log(n)/n}$ up to }\Lo(n^{-1})
\end{equation}
or
\begin{equation}
  \label{contbed2}
(P^{\SSs \rm di}_{n}) \mbox{ is  $K_1(n)$--concentrated right of $\hat y_n+ b\sqrt{k_2\log(n)/n}$ up to }\Lo(n^{-1})
\end{equation}
More precisely, if $\sup\psi \,<\,[>]\, -\inf \psi$, the maximal MSE is achieved by contaminations according to \eqref{contbed1}
[\eqref{contbed2}]. In case $\sup \psi=-\inf\psi$,  \eqref{contbed1} [\eqref{contbed2}] applies if
\begin{equation} \label{wokontlra}
\tilde v_1\,>\,[<]-\Tfrac{l_2}{4}\Big(\Tfrac{b^2}{v_0^2}(r^2+3)(1+\Tfrac{r}{\sqrt{n}}-\Tfrac{2r^2}{n})+3(1-\Tfrac{b^2}{v_0^2})\Big)
\end{equation}
If $\sup \psi=-\inf\psi$ and there is ``$=$'' in \eqref{wokontlra}, \eqref{contbed1} and \eqref{contbed2} generate the same
risk up to order $\Lo(n^{-1})$.
\end{ABC}
\end{Thm}
\begin{Rem}\rm\small\label{BigRem}
  \begin{ABC}
\item Curiously, although being of corresponding order, no $\rho_0$ [$\kappa_0$]-term shows up in the
correction term $A_1$ [$A_2$], which is probably due to the special loss function.
\item
As announced, for $r=0$, the approximation is one order faster than under contamination.
\item
{\bf  The maximal MSE on $\tilde {\cal Q}_n$ is always underestimated by f-o asymptotics,}
as maximality always forces $A_1$ to be non-negative.
\item
Let $Q_n^0$ be any distribution in $\tilde {\cal Q}_n$ attaining maximal risk
in Theorem~\ref{main}. Under symmetry or more specifically if $ l_2=v_1=\rho_0=0$,
\eqref{mainres} becomes
\begin{eqnarray}
&&n\,\Ew_{Q_n^0}[S_n^2\,]=
\left(  {r}^{2}{b}^{2}+{v_0}^{2}\right) \left(1+\Tfrac{r}{\sqrt{n}}\right)+\Tfrac{r}{\sqrt{n}}\,
 \left(b^2(1+r^2)\right)+\LO(n^{-1})\label{symloesa}
\end{eqnarray}
\item
In the ideal Gaussian location model (i.e.\ $r=0$), plugging in the (limiting) results for $c=0$ from section~\ref{illust}, the RHS of \eqref{symloesa}
 becomes
\begin{equation}
\frac{\pi}{2}\left(1+{n}^{-1}(\frac{\pi}{2}-\frac{5}{3})\right)+\Lo(n^{-1})\doteq 1.5708 (1-\Tfrac{0.0958}{n})+\Lo(n^{-1})
\end{equation}
suggesting an overestimation of the risk by the f-o asymptotics. This
 is to be compared to the result for the median for odd sample size from \citet{Ruck:03b}:
\begin{eqnarray}
n\,\Ew_{F^n}[{\rm Med}_n^2\,]&=&\frac{\pi}{2}\left(1+{n}^{-1}(\frac{\pi}{2}\!-\!2)\right)
+\Lo(n^{-1})
\doteq 1.5708 (1-\Tfrac{0.4292}{n})\!+\!\Lo(n^{-1})
\end{eqnarray}
so f-o asymptotics indeed overestimates the risk.
The difference of $\frac{\pi}{6\,n}$ is due to the failure of cond.~(C).
\item {\bf Relevance for the fixed neighborhood approach:} If you consider the fixed neighborhood approach (of radius $\ve$) and formally plug in $r=\ve\sqrt{n}$
into \eqref{mainres}, you obtain the following approximation for the unstandardized maximal MSE on the thinned out (fixed-radius) neighborhood:
\begin{eqnarray}
&&{\rm MSE}(S_n,\ve,\ve_0)=\ve^2b^2+ \ve^3\,[2\,{b}^{2}\pm l_2\,{b}^{3} \,]+
\ve^4\,\Big( (\Tfrac{5}{4}\,{l_2^2}+\Tfrac{1}{3}\,l_3 \,){b}^{4}
 \pm 3\,l_2\,{b}^{3} +3\,{b}^{2}\Big)+\nonumber\\
 &&\quad+
\frac{1}{n}v_0^2+\frac{\ve}{n}\Big[
{v_0}^{2}\,\Big( \pm(4\,\tilde v_1+3\,l_2 \,)b+1 \Big)+{b}^{2}\Big] +
\frac{\ve^2}{n} \Big[ 5 \,{b}^{2} \pm 3\,l_2\,{b}^{3}+\nonumber\\
&&\quad +{v_0}^{2}\Big( (3\tilde v_2+3{\tilde v_1}^{2}+\Tfrac{15}{2}{l_2^2}+2l_3+
 12\tilde v_1l_2 ){b}^{2}+1\pm (8\tilde v_1+6l_2 )\,b \Big)
\,\Big]+R_n
\end{eqnarray}
for some remainder $R_n$ the order of which however is uncertain; it should be valid for small $\ve$, and is at least of order $\LO(1/n^2)+\LO(\ve^5)$.
These terms once more show that for the fixed-neighborhood approach, already for moderate sample sizes, bias becomes
dominant, i.e.; in our case, we end up with the median as optimal procedure.
  \end{ABC}
\end{Rem}
%
\subsection{Cross-checks}
\subsubsection{Check with results by Fraiman et al.} \label{xcheckfrai}
In the symmetric case, the first cross check comes with the asymptotic  formula for variance ${\rm asVar}(\psi)$ and (maximal)
bias $B(\psi):={\rm asBias}(\psi)$ as to be found in \citet{F:Y:Z:01}, where we have to identify  $\ve=r/\sqrt{n}$.
Here,  ${\rm asBias}(\psi)/\sqrt{n}$ is defined as zero $\beta$ of $\beta\mapsto (1-\ve) \int\psi_{\beta}\, dF+\ve b$,
and 
  ${\rm asVar}(\psi):=V_1/V_2^2$ for $V_1=(1-\ve)\int \psi_{B(\psi)}^2\, dF+\ve b^2$ and $V_2=(1-\ve) \int \dot\psi_{B(\psi)}\, dF$.
%
Assuming that $\int \dot\psi_{B(\psi)}\, dF=L'(B(\psi))$ and using that
$\int\psi_{B(\psi)}\, dF=-B(\psi)+\Lo(B^2))$, $
\int\psi_{B(\psi)}^2\, dF=V(B(\psi))^2+L(B(\psi))^2=v_0^2(1+\Lo(B))$,
$L'(B(\psi))^2=-1+\Lo(B)$,
we get that
\begin{eqnarray}
{\rm asBias}(\psi)&=&\sqrt{n}\, b\ve (1+\ve+\Lo(\ve))=rb(1+\Tfrac{r}{\sqrt{n}}+\Lo(n^{-1/2}))\\
{\rm asVar}(\psi)&=&(1+\ve)v_0^2+\ve b+\Lo(\ve)=v_0^2+\Tfrac{r}{\sqrt{n}}(v_0^2+b)+\Lo(n^{-1/2})
\end{eqnarray}
and hence ---in accordance with formula \eqref{mainres}---
\begin{eqnarray}
{\rm asMSE}(\psi)
&=&(v_0^2+r^2b^2)(1+\Tfrac{r}{\sqrt{n}})+\Tfrac{r}{\sqrt{n}}b^2(1+r^2)+\Lo(n^{-1/2})
\end{eqnarray}
\subsubsection{Check with higher order asymptotics for the median}
The second check comes with the higher asymptotics for the median from \citet{Ruck:03b}.
In a first step, we assume that with $f_0>0$ and some $\delta\in(0,1]$,
\begin{equation}
  f(t)=f_0+f_1t+\LO(t^{1+\delta})
\end{equation}
As for the median, $\psi_{\rm\SSs Med} = \sign(x)/({2\,f_0})$, we have
$v_0=b=\Tfrac{1}{2f_0}$ and $\ve_0=1/2$.
For the moment we ignore the fact, that conditions~(C)/(C') are not fulfilled.
Easy calculations give
$l_2=-f_1/f_0$, $\tilde v_1=0$, $\rho_0=0$,
so that with our formula \eqref{mainres} we obtain for odd sample size $n$
\begin{eqnarray}
R_n(\psi_{{\rm\SSs Med}_n},r,\Tfrac{1}{2})\!&\!=\!&\!\frac{1}{4f_0^2}\Big((1+ r^{2})\big[1+\Tfrac{2r}{\sqrt{n}}\big]-\Tfrac{r}{\sqrt{n}}\frac{f_1}{2f_0^2}(r^2+3) \Big)+\Lo(n^{-1/2})
\end{eqnarray}
in complete agreement with \citet{Ruck:03b}.
%
%
As a next step we compare this to t-o asymptotics to be obtained
by \eqref{mainres}---again ignoring condition (C). We get
$l_3=-f_2/f_0$, $\tilde v_2=-4f_0^2$, $\rho_1=4f_0$,
and hence for odd sample size $n$, after some reordering
\begin{eqnarray}
\!&&\!R_n(\psi_{{\rm\SSs Med}_n},r,\Tfrac{1}{2})\stackrel{?}{=}\Lo(\Tfrac{1}{n})+\frac{1}{4 f_0^2}\Bigg\{
(1+r^2)+\Tfrac{r}{\sqrt{n}}\Big(2(1+r^2)+\Tfrac{r^2+3}{2}\frac{|f_1|}{f_0^2}\Big)+\nonumber\\
\!&&\!\quad +\Tfrac{1}{n}\Big(\fbox{$\Tfrac{4}{3}$}-3+3r^2+3r^4+\Tfrac{3r^2(3+r^2)}{2}\frac{|f_1|}{f_0^2}- 
\Tfrac{3+6r^2+r^4}{12}\frac{f_2}{f_0^3}+\Tfrac{5(3+6r^2+r^4)}{16}\frac{f_1^2}{f_0^4}\Big)\Bigg\}\label{MALL}
\end{eqnarray}
and it is just the framed term $\frac{4}{3}$, which is coming in as $\frac{2}{3} \rho_1 v_0$ from \eqref{A2defr},
which causes a difference to the result of \citet{Ruck:03b}, where we get the value $1$ instead. This discrepancy, however,
is in fact due to the failure of condition (C), because Theorem~\ref{berry2}, which we need to prove  \eqref{mainres}, is not
available in this case.
\section{Relations to other approaches}\label{relsec}\req
%
%
Of course the idea of assessing the quality / speed of convergence of
CLT-type arguments by means of higher order asymptotics is common
in Mathematical Statistics, cf.\  among others \citet{Ib:Li:71}, \citet{Bh:Ra:76}, \citet{Pfa:85},
\citet{Hal:92}, \citet{Ba-N:C:94} and \citet{Ta:Ka:00}.\\
Asymptotic expansions of the moments of statistical estimators ---like MSE in our case--- have already
been studied by \citet{Gus:76} and \citet{Pfaf:Diss}; both approaches, however, only
consider the ideal model, and work with pointwise expansions of the likelihood.\\
Also the idea to improve convergence by means of saddlepoint techniques and
conjugate densities, respectively, has been a large success in this context,
cf.\  \citet{Da:54}, \citet{Ham:73}, \citet{Fi:Ro:90}.\\
Our approach is simpler in the sense that instead of approximating
the c.d.f.\ or the density of our procedures 
on the whole range of arguments, we directly approximate our risk. Doing so, we do not run into problems of bad
approximations in the tails of a distribution, because all that is interesting for
our  risk will occur within a (decreasing) compact; using saddlepoint techniques,
we would have to solve the saddlepoint-equation for a grid of evaluation points
$t_i$ to get an accurate estimate for the density which makes the corresponding solution less explicite than
ours.\\
Even more important, note that a highly accurate approximation of the distribution
of the M-estimator would not suffice to enforce uniform convergence of the MSE, which
was the reason for our modification of the neighborhoods \eqref{modifdefi}.
Also, contrary to ``usual'' small
sample asymptotics,  by our approach no particular contamination has to be assumed right from
the beginning but we rather identify a least favorable one within the proof.
%
\\
In the setup of saddlepoint-approximations,
one would apply \citet[Theorem~4.3]{Fi:Ro:90} 
which at least covers the Hampel-type solutions. 
The pointwise formulation of assumption A4.2 therein, i.e.; 
there exists an open subset $U\subset \R$, such that 
(i) for each $\theta \in \R$, $F(U-\theta)=1$ 
and (ii) $D\psi$, $D^2\psi$, $D^3\psi$ exist on $U$,
seems problematic, however,
as it allows for pathological $\psi$-functions defined similar to the Cantor
distribution function (while $F$ may be something like ${\cal N}(0,1)$),
for which the interchange of differentiation and integration becomes awkward.
As may be read off from \eqref{mainres}, in the ideal model, as for the saddlepoint approach, we, too,  get an expansion of order $1/n$,
a fact, which is {\em not\/} due to symmetry of $\Lambda$ and/or $\psi$! So in fact
we get the same approximation quality as with the saddlepoint approach ---indeed, by the Taylor-expansion step
in section~\ref{neuexp}, we extract an argument to be expanded from the exponential, which also is an idea behind
the saddlepoint approximation, cf.\  \citet[p.~26]{Fi:Ro:90}. On the other hand, even in the restricted neighborhoods of
\eqref{modifdefi},
it is not clear to the present author,
if in general, the saddlepoint approximation holds uniformly in $t$, so it is not clear,
whether an improved approximation for the density will result in a better approximation of the risk.
A detailed empirical and numerical investigation of such questions is contained in \citet{Ru:Ko:04}.
%
%
%

\section{A simulation study and numerical evaluations} \label{simsec}\req
Before starting with the theoretical findings we summarize the results of a simulation study that actually lead us to the closer
examination of the higher order expansions of the MSE.
\subsection{Simulation design}
Under {\tt R 2.11.0}, c.f.\ \citet{RMANUAL}, we simulated $M=10000$ runs of sample size $n=5,10,30,50,100$ in the ideal location model
${\cal P}={\cal N}(\theta,1)$ at $\theta=0$. In a contaminated situation, we used observations stemming from
\begin{equation}
Q_n={\cal L}\{[(1-U_i)X_i^{\SSs \rm id}+U_i X_i^{\SSs \rm di}]_i\,\Big|\,\sum U_i\leq \ulcorner n/2 \urcorner -1\,\}
\end{equation}
for $U_i\iid {\rm Bin}(1,r/\sqrt{n})$, $X_i^{\SSs \rm id}\iid{\cal N}(0,1)$, $X_i^{\SSs \rm di}\iid \Jc_{\{100\}}$
all stochastically independent and for contamination radii $r=0.1, 0.25, 0.5, 1.0$. \\
As estimators we considered the median (with the mid-point variant for even sample size), and M-estimators to Hampel-type
ICs $\eta_c$ of form \eqref{HK1} with clipping heights $c=0.5, 0.7, 1, 1.5, 2$ and $c_0(r)$, the f-o-o clipping height according to
\eqref{HK3}. All empirical ${\rm MSE}$'s come with asymptotic  $95\%$--confidence intervals,
which are based on the CLT for the variables \begin{equation}%
\overline{\rm empMSE}_n=\Tfrac{n}{10000}\sum\nolimits_j [S_n({\rm sample}_j)]^2
\end{equation}
Note that with respect to 
\eqref{contbed1}/\eqref{contbed2},
and the considered estimators, a contamination point $100$ will
largely suffice to attain the maximal MSE on $\tilde {\cal Q}_n$.
\subsection{Numerical evaluations}\label{numsec}
By means of relations~\eqref{mest2} we may reduce the problem of finding the exact distribution of our M-estimators
to the calculation of the ``exact'' distribution of $\sum_i \psi(X_i)$. For this purpose, we may apply the general
convolution algorithm for arbitrarily distributed real-valued random variables
introduced in \citet{R:K:10:FFT}. This algorithm is based on FFT resp.\ discrete Fourier Transformation (DFT) and is implemented in
{\tt R} within the package {\tt distr} available on {\tt CRAN}, see \citet{RKSC:06}, \citet{RMandistr}.\\
In \citet{Ru:Ko:04}, to increase accuracy for M-estimators to Hampel IC's, we extend our algorithm from {\tt distr} to (a) better cope with
mass points in $\pm b$ and (b) to calculate the ``exact'' finite-sample maximum MSE on $\tilde {\cal Q}_n$.
Here we confine ourselves to attach extra columns ``numeric'' to the
following tables summarizing our simulation. ``numeric'' will then stand for application of
Algorithm~C respectively Algorithm~D from \citet{Ru:Ko:04}.\\
More specifically, for ``exact'' terms, as worked out in Algorithm~C (ibid.), we have to take into account that
after conditioning w.r.t.\ the event that the number of contaminations $K$ in the sample is less than
half the sample size, the switching variables $U_i$ from \eqref{Uidef} no longer are independent.
So we may only apply the FFT-based Algorithm from \citet{Ru:Ko:04} to an absolutely continuous inner part and have to calculate the rest
by explicitly summing up the events---for details see the cited reference and the {\tt R}-program available on the web-page to this article.

%
On the other side, Algorithm~D uses the fact that by the exponential negligibility shown in
subsection~\ref{expneg}, the dependency of the $U_i$ may be ignored for $n$ sufficiently large---in our case this
was possible for $n\geq 30$, moderate radius $r$ and robust clipping height $c$.
Then, we simply may determine the corresponding convolutions of the corresponding
distributions of the summands directly by Algorithm~4.4 from \citet{R:K:10:FFT}.\\
To demonstrate the negligibility, for $n\leq 30$, we calculate both ``exact'' terms (Algortihm~C) and those
obtained by superposition of the a.c. part and the random walk, ignoring all mass points of the law of the sum (Algortihm~D).
\subsection{Results}
A more detailed account of the results of the simulation study in tables may be found at
the web-page to this article. 
Here we only present some few results which led to the subsequent investigation.
\subsubsection{Fixed procedure, fixed radius}
To get an idea of the speed of the convergence of the MSE  to its asymptotic  values,
we consider  the {\tt H07}-estimator
from \citet{A:B:H:H:R:T:72}, i.e.\ the M-estimator to $\eta_{0.7}$ at $r=0.1$ and at $r=0.5$ for different sample sizes $n$.\\
The simulated empirical risk comes with an (empirical) $95\%$ confidence interval and is compared to the corresponding numerical approximations and
to the f-o, s-o, and t-o asymptotics from Theorem~\ref{main}.
Corresponding tables for the f-o-o M-estimator to $\eta_{c_0}$ may be drawn from the web-page to this article.
The results are tabulated in Tables~\ref{tabel2}/\ref{tabel4}.
\begin{table}
\caption{\label{tabel2}{emp., num., and as.\   }$\boldmath {\rm MSE}$ at $\boldmath r=0.1$, $c=0.7$}
\begin{center}
\begin{tabular}{rr||r@{\hspace{0.5em}\hfill$[$}c@{;}c@{$\,]\;$}|rr|rrr|}
\multicolumn{2}{l||}{$n$/}&\multicolumn{3}{c|}{simulation}&\multicolumn{2}{c|}{numeric}
&\multicolumn{3}{c|}{asymptotics}\\
\multicolumn{2}{r||}{situation}&  \multicolumn{1}{c}{$\bar S_n$}& \multicolumn{1}{c}{[low;}&  \multicolumn{1}{c|}{up]}
&\multicolumn{1}{c}{\footnotesize Algo~C}&\multicolumn{1}{c|}{\footnotesize Algo~D}&\multicolumn{1}{c}{$n^0$}&
 \multicolumn{1}{c}{$n^{-1/2}$}& \multicolumn{1}{c|}{$n^{-1}$}\\
\hline
                               &${\rm \Ss id}  $& 1.147 & 1.114& 1.179 &  1.172               & 1.168 & 1.187 & 1.187 & 1.169\\
\raisebox{1.5ex}[-1.5ex]{$  5$}&${\rm \Ss cont}$& 1.403 & 1.359& 1.447 &  1.434               & 1.535 & 1.205 & 1.342 & 1.345\\
\hline
                               &${\rm \Ss id}  $& 1.179 & 1.139& 1.205 &  1.177               & 1.174 & 1.187 & 1.187 & 1.178\\
\raisebox{1.5ex}[-1.5ex]{$ 10$}&${\rm \Ss cont}$& 1.331 & 1.292& 1.369 &  1.327               & 1.326 & 1.205 & 1.302 & 1.303\\
\hline
                               &${\rm \Ss id}  $& 1.209 & 1.175& 1.242 &  1.183               & 1.180 & 1.187 & 1.187 & 1.184\\
\raisebox{1.5ex}[-1.5ex]{$ 30$}&${\rm \Ss cont}$& 1.301 & 1.264& 1.337 &  1.265               & 1.262 & 1.205 & 1.261 & 1.261\\
\hline
                               &${\rm \Ss id}  $& 1.192 & 1.158& 1.225 &\multicolumn{1}{c}{--}& 1.181 & 1.187 & 1.187 & 1.185\\
\raisebox{1.5ex}[-1.5ex]{$ 50$}&${\rm \Ss cont}$& 1.250 & 1.214& 1.285 &\multicolumn{1}{c}{--}& 1.247 & 1.205 & 1.248 & 1.249\\
\hline
                               &${\rm \Ss id}  $& 1.161 & 1.128& 1.193 &\multicolumn{1}{c}{--}& 1.182 & 1.187 & 1.187 & 1.186\\
\raisebox{1.5ex}[-1.5ex]{$100$}&${\rm \Ss cont}$& 1.212 & 1.178& 1.246 &\multicolumn{1}{c}{--}& 1.232 & 1.205 & 1.236 & 1.236
\end{tabular}
\end{center}
\end{table}
\begin{table}
\caption{\label{tabel4}{emp., num., and as.\   }$\boldmath {\rm MSE}$ at $\boldmath r=0.5$, $c=0.7$}
\begin{center}
\begin{tabular}{rr||r@{\hspace{0.5em}\hfill$[$}c@{;}c@{$\,]\;$}|rr|rrr|}
\multicolumn{2}{l||}{$n$/}&\multicolumn{3}{c|}{simulation}&\multicolumn{2}{c|}{numeric}
&\multicolumn{3}{c|}{asymptotics}\\
\multicolumn{2}{r||}{situation}&  \multicolumn{1}{c}{$\bar S_n$}& \multicolumn{1}{c}{[low;}&  \multicolumn{1}{c|}{up]}
&\multicolumn{1}{c}{\footnotesize Algo~C}&\multicolumn{1}{c|}{\footnotesize Algo~D}&\multicolumn{1}{c}{$n^0$}&
 \multicolumn{1}{c}{$n^{-1/2}$}& \multicolumn{1}{c|}{$n^{-1}$}\\
\hline
                               &${\rm \Ss id}  $& 1.166 & 1.134& 1.199 &   1.172              & 1.168  & 1.187 & 1.187 & 1.169\\
\raisebox{1.5ex}[-1.5ex]{$  5$}&${\rm \Ss cont}$& 2.989 & 2.892& 3.087 &   3.016              & 12.491 & 1.647 & 2.529 & 3.103\\
\hline
                               &${\rm \Ss id}  $& 1.191 & 1.157& 1.224 &   1.177              & 1.174  & 1.187 & 1.187 & 1.178\\
\raisebox{1.5ex}[-1.5ex]{$ 10$}&${\rm \Ss cont}$& 2.934 & 2.836& 3.032 &   2.840              & 4.820  & 1.647 & 2.271 & 2.557\\
\hline
                               &${\rm \Ss id}  $& 1.194 & 1.161& 1.227 &   1.183              & 1.180  & 1.187 & 1.187 & 1.184\\
\raisebox{1.5ex}[-1.5ex]{$ 30$}&${\rm \Ss cont}$& 2.183 & 2.119& 2.247 &   2.167              & 2.167  & 1.647 & 2.007 & 2.102\\
\hline
                               &${\rm \Ss id}  $& 1.165 & 1.133& 1.197 &\multicolumn{1}{c}{--}& 1.181  & 1.187 & 1.187 & 1.185\\
\raisebox{1.5ex}[-1.5ex]{$ 50$}&${\rm \Ss cont}$& 1.946 & 1.893& 1.998 &\multicolumn{1}{c}{--}& 2.008  & 1.647 & 1.926 & 1.983\\
\hline
                               &${\rm \Ss id}  $& 1.192 & 1.159& 1.226 &\multicolumn{1}{c}{--}& 1.182  & 1.187 & 1.187 & 1.186\\
\raisebox{1.5ex}[-1.5ex]{$100$}&${\rm \Ss cont}$& 1.894 & 1.844& 1.944 &\multicolumn{1}{c}{--}& 1.879  & 1.647 & 1.844 & 1.873
\end{tabular}
\end{center}
\end{table}
%
In 
Table~\ref{tabel5} 
we consider the relative MSE, calculated as the quotient
${\rm MSE}(c,r)/{\rm MSE}(c_0(r),r)$. This is a natural
expression to compare the efficiency of different procedures. We compare the empirical terms from the
simulation to the corresponding numerical approximations and to the asymptotic  terms derived by means of Theorem~\ref{main}.
We already recognize a very good
approximation down to very small sample sizes.
\begin{table}
\parbox[t]{12cm}
{\caption[emp., num., and as.\ ${\rm relMSE}$ at $r=0.1,0.5$, $c=0.7$]%
{\label{tabel5}\parbox[t]{11cm}{emp., num., and as.\ ${\rm relMSE}$ at $\boldmath r=0.1,0.5$, $c=0.7$
 relative to $\Var[\bar X_n]$ for ${\rm \Ss id}$ and ${\rm MSE}(c_0(r))$ for ${\rm \Ss cont}$}}%
}
\begin{center}
\begin{tabular}{rr||r|r|rr||r|r|rr||}
&&\multicolumn{4}{c||}{$r=0.1$}&\multicolumn{4}{c||}{$r=0.5$}\\
\multicolumn{2}{l||}{\raisebox{1ex}[-1ex]{$n$/}}&\multicolumn{1}{c|}{sim}&\multicolumn{1}{c|}{num}&\multicolumn{2}{c||}{asymptotics}
&\multicolumn{1}{c|}{sim}&\multicolumn{1}{c|}{num}&\multicolumn{2}{c||}{asymptotics}\\
\multicolumn{2}{r||}{\raisebox{1ex}[-1ex]{situation}}&&\multicolumn{1}{c|}{\footnotesize ex/$\ast$}&%
\multicolumn{1}{c}{$n^0$}&\multicolumn{1}{c||}{$n^{-1/2}$}&%
&\multicolumn{1}{c|}{\footnotesize ex/$\ast$}&\multicolumn{1}{c}{$n^0$}&\multicolumn{1}{c||}{$n^{-1/2}$}\\
\hline
                               &${\rm \Ss id}  $& 1.161 & $1.163\hphantom{^{\SSs \ast}}$& 1.173 & 1.173 & 1.038 & $1.042\hphantom{^{\SSs \ast}}$& 1.041 & 1.041\\
\raisebox{1.5ex}[-1.5ex]{$  5$}&${\rm \Ss cont}$& 1.003 & $0.956\hphantom{^{\SSs \ast}}$& 1.143 & 1.039 & 0.992 & $0.978\hphantom{^{\SSs \ast}}$& 1.006 & 0.989\\
\hline
                               &${\rm \Ss id}  $& 1.167 & $1.166\hphantom{^{\SSs \ast}}$& 1.173 & 1.173 & 1.037 & $1.041\hphantom{^{\SSs \ast}}$& 1.041 & 1.041\\
\raisebox{1.5ex}[-1.5ex]{$ 10$}&${\rm \Ss cont}$& 1.049 & $1.029\hphantom{^{\SSs \ast}}$& 1.143 & 1.065 & 0.993 & $0.977\hphantom{^{\SSs \ast}}$& 1.006 & 0.992\\
\hline
                               &${\rm \Ss id}  $& 1.174 & $1.170\hphantom{^{\SSs \ast}}$& 1.173 & 1.173 & 1.037 & $1.041\hphantom{^{\SSs \ast}}$& 1.041 & 1.041\\
\raisebox{1.5ex}[-1.5ex]{$ 30$}&${\rm \Ss cont}$& 1.094 & $1.086\hphantom{^{\SSs \ast}}$& 1.143 & 1.095 & 0.994 & $0.993\hphantom{^{\SSs \ast}}$& 1.006 & 0.997\\
\hline
                               &${\rm \Ss id}  $& 1.160 & $1.169^{\SSs \ast}$           & 1.173 & 1.173 & 1.038 & $1.041^{\SSs \ast}$           & 1.041 & 1.041\\
\raisebox{1.5ex}[-1.5ex]{$ 50$}&${\rm \Ss cont}$& 1.096 & $1.096^{\SSs \ast}$           & 1.143 & 1.105 & 0.996 & $0.995^{\SSs \ast}$           & 1.006 & 0.999\\
\hline
                               &${\rm \Ss id}  $& 1.180 & $1.170^{\SSs \ast}$           & 1.173 & 1.173 & 1.044 & $1.041^{\SSs \ast}$           & 1.041 & 1.041\\
\raisebox{1.5ex}[-1.5ex]{$100$}&${\rm \Ss cont}$& 1.122 & $1.110^{\SSs \ast}$           & 1.143 & 1.116 & 0.999 & $0.999^{\SSs \ast}$           & 1.006 & 1.001
\end{tabular}
\end{center}
\end{table}%
\subsubsection{Fixed procedure, fixed sample size}
In order  to study the effect of the radius on the quality of the approximation,
we consider the 
M-estimator to $\eta_{0.5}$ at sample size $n=30$ at varying radii.
The results are tabulated in Table~\ref{tabel6}. 
The simulations and the numeric values clearly show that with increasing radius, the approximation quality of f-o
asymptotics decreases, which is conformal to the infinitesimal character of our neighborhoods.
A corresponding table for the more liberal M-estimator to $\eta_{2}$ at sample
size $n=50$ may be drawn from the web-page. 
\\
%
\begin{table}
\caption{\label{tabel6}{emp., num., and as.\ }$\boldmath {\rm MSE}$ at $n=30$, $c=0.5$}
\begin{center}
\begin{tabular}{r||r@{\hspace{0.5em}\hfill$[$}c@{;}c@{$\,]\;$}|rr|rrr|}
&\multicolumn{3}{c|}{simulation}&\multicolumn{2}{c|}{numeric}&\multicolumn{3}{c|}{asymptotics}\\
\multicolumn{1}{l||}{\raisebox{1.5ex}[-1.5ex]{$r$}}&  \multicolumn{1}{c}{$\bar S_n$}&%
\multicolumn{1}{c}{[low;}&  \multicolumn{1}{c|}{up]}& \multicolumn{1}{c}{\footnotesize Algo~C}&%
\multicolumn{1}{c|}{\footnotesize Algo~D}&\multicolumn{1}{c}{$n^0$}&%
  \multicolumn{1}{c}{$n^{-1/2}$}& \multicolumn{1}{c|}{$n^{-1}$}\\
\hline
0.00   & 1.272 & 1.237& 1.307 & 1.259 & 1.256 & 1.263 & 1.263 & 1.259\\
0.10   & 1.374 & 1.336& 1.413 & 1.337 & 1.335 & 1.280 & 1.334 & 1.334\\
0.25   & 1.545 & 1.502& 1.588 & 1.545 & 1.542 & 1.588 & 1.514 & 1.532\\
0.50   & 2.204 & 2.139& 2.268 & 2.189 & 2.187 & 1.689 & 2.037 & 2.128\\
1.00   & 5.362 & 5.219& 5.505 & 5.238 & 5.265 & 2.967 & 4.132 & 4.652
\end{tabular}
\end{center}
\end{table}
%
%
%
\subsubsection{Fixed radius, fixed sample size}
In this paragraph we want to compare M-estimators to different clipping heights and see whether the choice of $c_0$ may also
be considered reasonable for moderate $n$. To this end, we consider the situation 
$r=0.25$ and $n=30$.
The results are tabulated in Tables~\ref{tabel10} and \ref{tabel11}. 
The simulations already indicate that the answer should be affirmative.
The numeric and asymptotic  values for the median are taken from \citet{Ruck:03b}.
Corresponding tables to the situation $r=0.5$ and $n=100$ are on the web-page.
\\
\begin{table}
\caption{\label{tabel10}{emp., num., and as.\ }$\boldmath {\rm MSE}$ at $n=30$, $r=0.25$}
\begin{center}
\begin{tabular}{rr||r@{\hspace{0.5em}\hfill$[$}c@{;}c@{$\,]\;$}|r|rrr|}
\multicolumn{2}{l||}{estimator/}&\multicolumn{3}{c|}{simulation}&\multicolumn{1}{c|}{num}&%
\multicolumn{3}{c|}{asymptotics}\\
\multicolumn{2}{r||}{situation}&  \multicolumn{1}{c}{$\bar S_n$}& \multicolumn{1}{c}{[low;}&  \multicolumn{1}{c|}{up]}&
\multicolumn{1}{c|}{ex}&\multicolumn{1}{c}{$n^0$}&  \multicolumn{1}{c}{$n^{-1/2}$}& \multicolumn{1}{c|}{$n^{-1}$}\\
\hline
                               &${\rm \Ss id}           $& 1.492 & 1.451& 1.532 & 1.501 & 1.571 & 1.571 & 1.496\\
\raisebox{1.5ex}[-1.5ex]{$\rm Med$}&${\rm \Ss cont}     $& 1.786 & 1.736& 1.835 & 1.779 & 1.669 & 1.821 & 1.767\\
\hline
                               &${\rm \Ss id}           $& 1.250 & 1.216& 1.284 & 1.259 & 1.263 & 1.263 & 1.259\\
\raisebox{1.5ex}[-1.5ex]{$c=0.5$}&${\rm \Ss cont}       $& 1.545 & 1.502& 1.588 & 1.545 & 1.369 & 1.514 & 1.532\\
\hline
                               &${\rm \Ss id}           $& 1.092 & 1.062& 1.122 & 1.105 & 1.107 & 1.107 & 1.105\\
\raisebox{1.5ex}[-1.5ex]{$c=1.0$}&${\rm \Ss cont}       $& 1.433 & 1.393& 1.473 & 1.440 & 1.241 & 1.402 & 1.425\\
\hline
                               &${\rm \Ss id}           $& 0.991 & 0.963& 1.018 & 1.010 & 1.010 & 1.010 & 1.010\\
\raisebox{1.5ex}[-1.5ex]{$c=2.0$}&${\rm \Ss cont}       $& 1.611 & 1.566& 1.656 & 1.633 & 1.285 & 1.556 & 1.604\\
\hline
                               &${\rm \Ss id}           $& 1.035 & 1.006& 1.063 & 1.051 & 1.139 & 1.053 & 1.052\\
\raisebox{1.5ex}[-1.5ex]{$c=c_0=1.3393$}&${\rm \Ss cont}$& 1.438 & 1.398& 1.479 & 1.452 & 1.220 & 1.405 & 1.434
\end{tabular}
\end{center}
\end{table}
%
\begin{table}
\parbox[t]{12cm}
{\caption[{emp., num., and as.\ }${\rm relMSE}$ at $n=30$, $r=0.25$]%
{\label{tabel11}\parbox[t]{11cm}{{emp., num., and as.\ }${\rm relMSE}$ at $n=30$, $r=0.25$ relative to
$\Var[\bar X_n]$ for ${\rm \Ss id}$ and ${\rm MSE}(c_0(r))$ for ${\rm \Ss cont}$,  $c_0(r)=1.3393$}}%
}
\begin{center}
\begin{tabular}{rr||r|r|rr|}
\multicolumn{2}{l||}{estimator/}&\multicolumn{1}{c|}{simulation}&\multicolumn{1}{c|}{numeric}&\multicolumn{2}{c|}{asymptotics}\\
\multicolumn{2}{r||}{situation}&&\multicolumn{1}{c|}{ex}&\multicolumn{1}{c}{$n^0$}&  \multicolumn{1}{c|}{$n^{-1/2}$}\\
\hline
                               &${\rm \Ss id}           $& 1.435 & 1.427 & 1.379 & 1.379\\
\raisebox{1.5ex}[-1.5ex]{$\rm Med$}&${\rm \Ss cont}     $& 1.241 & 1.224 & 1.320 & 1.263\\
\hline
                               &${\rm \Ss id}           $& 1.202 & 1.197 & 1.199 & 1.198\\
\raisebox{1.5ex}[-1.5ex]{$c=0.5$}&${\rm \Ss cont}       $& 1.073 & 1.064 & 1.077 & 1.068\\
\hline
                               &${\rm \Ss id}           $& 1.051 & 1.051 & 1.051 & 1.051\\
\raisebox{1.5ex}[-1.5ex]{$c=1.0$}&${\rm \Ss cont}       $& 0.995 & 0.991 & 0.998 & 0.994\\
\hline
                               &${\rm \Ss id}           $& 0.953 & 0.960 & 0.959 & 0.960\\
\raisebox{1.5ex}[-1.5ex]{$c=2.0$}&${\rm \Ss cont}       $& 1.119 & 1.125 & 1.107 & 1.119
\end{tabular}
\end{center}
\end{table}
\subsubsection{Relative error compared to numerically exact risk}
A closer look onto the relative error of our higher order asymptotics w.r.t.\ the numerically exact risk ${\rm MSE}_n$ is provided by figure~\ref{fig1}.
A zoom-in for $n\ge 16$ is available on the web-page. 
Indeed for all investigated
radii $r=0.00$, $0.10$, $0.25$, $1.00$, the relative error of our asymptotic  formula w.r.t.\ the corresponding numeric figures is quickly
decreasing in absolute value in $n$; also, we notice that we have a certain oscillation between odd and even sample sizes for very
small $n$ which is explained by the fact that for even $n$ there may be ties. By Lemma~\ref{Kohllem}, the contribution of these ties
to the risk is however decaying exponentially in $n$. \\In table~\ref{tabel2n}, 
we have determined the smallest sample size $n_0$ such that for $n\ge n_0$ the relative error using first to third order asymptotics
for approximating ${\rm MSE}_n(\psi_c)$ to $c=0.7$ is smaller than $1\%$ resp.\ $5\%$ which shows that for
$r\leq 0.5$ we need no more than $25$ $(60)$ observations to stay within an error corridor of $5\%$ ($1\%$) in t-o asymptotics.
For f-o asymptotics, however  we need considerable sample sizes for reasonable approximations unless the radius is rather small.\\
The figures in this table are to be taken ``cum grano salis'' due to numerical inaccuracies in ${\rm MSE}_n$ w.r.t.\ the exact risk
of order $1{\Ss\rm E}-5$ which may result in a deviation from the ``real'' $n_0$ of $\pm 2$ for $n_0<200$.
\begin{figure}
  \begin{center}
    \includegraphics[width=13cm,height=9.cm]{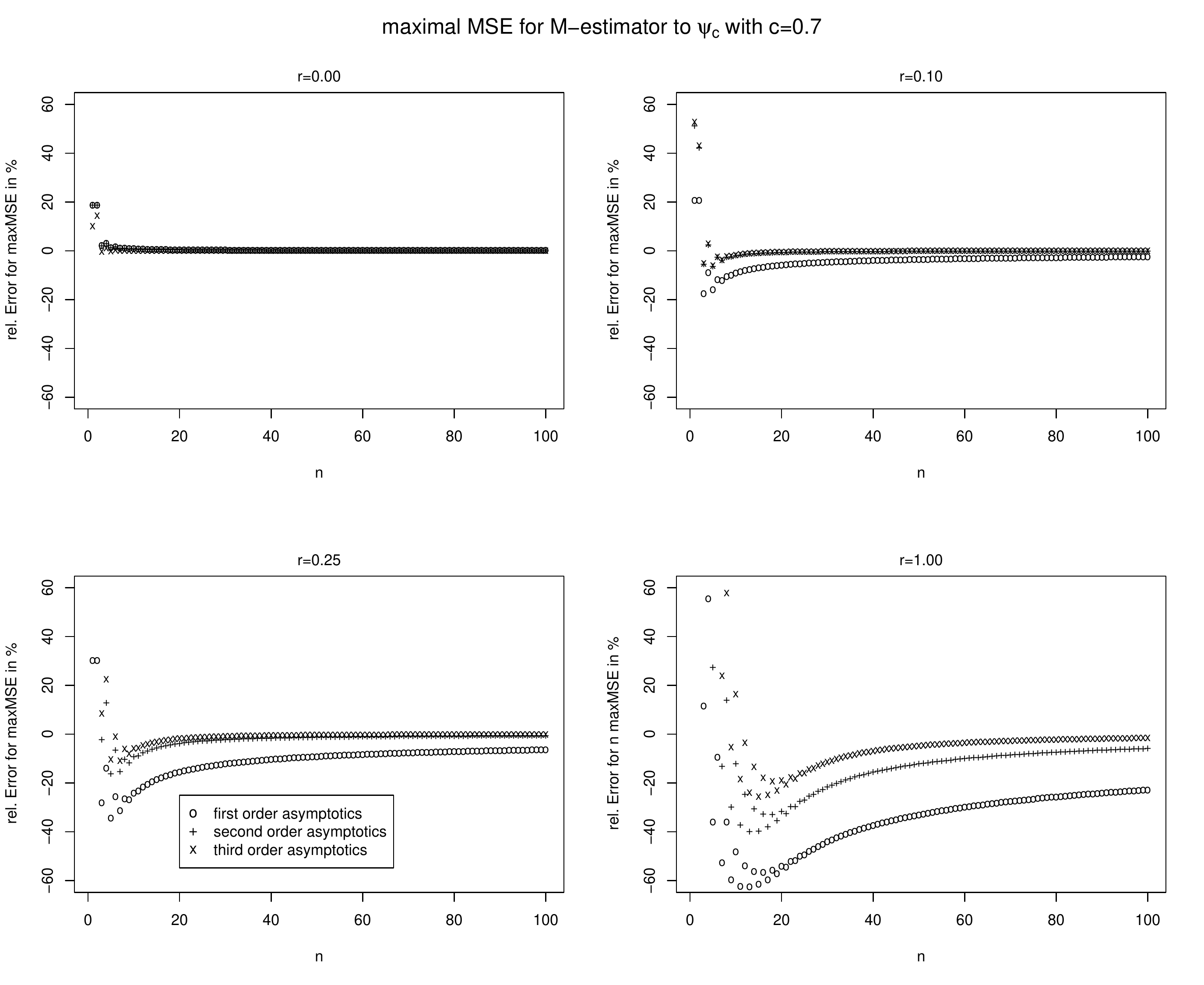}
    \caption{\label{fig1}{\rm\small The mapping $n\mapsto {\rm rel.error}({\rm MSE}_n(\psi_c))$
    for $c=0.7$ and $F={\cal N}(0,1)$.}}
  \end{center}
\end{figure}
%
\begin{table}
\parbox[t]{12cm}
{\caption[Minimal $n_0$ such that for $n\ge n_0$ the relative error is smaller than $1\%$ resp.\ $5\%$]%
{\label{tabel2n}\parbox[t]{11cm}{\rm\small Minimal $n_0$ such that for $n\ge n_0$ the relative error using first to third order asymptotics
for approximating ${\rm MSE}_n(\psi_c)$ for $c=0.7$ is smaller than $1\%$ resp.\ $5\%$}}%
}
\begin{center}
\renewcommand{\arraystretch}{1.1}
\begin{tabular}{l|l||r|r|r|r|r|}
${\rm rel.err}$& order&$r=0.00$      & $r=0.10$   & $r=0.25$    & $r=0.50$     & $r=1.00$     \\
\hline
1\%&1st order asy.    & $9$          & $>640^\ast$& $>3927^\ast$& $>14425^\ast$& $>49220^\ast$\\
   &2nd order asy.    & $9$          & $15$       & $60$        &   $196$      &  $>580^\ast$  \\
   &3rd order asy.    & $5$          & $15$       & $30$        &    $59$      &   $146$      \\
\hline
5\%&1st order asy.    & $3$          & $28$       & $162$       &   $>590^\ast$& $>1995^\ast$ \\
   &2nd order asy.    & $3$          & $6$        & $17$        &    $43$      &   $119$      \\
   &3rd order asy.    & $3$          & $6$        & $12$        &    $23$      &    $49$      \\
\end{tabular}
\renewcommand{\arraystretch}{1}
\end{center}
\begin{footnotesize}
  $\ast$: for $n>200$ computation of  ${\rm MSE}_n$ gets too expensive in time;
  instead we use  the the corresponding t-o figure.
  Assuming an error of t-o asymptotics of order
  $\LO(n^{-3/2})$, a corresponding regression onto the error term gives estimates for the regression coefficient to
  the term $n^{-3/2}$  of about  $-50$,  $-166$, $-534$, and $-1940$ for $r=0.1$, $0.25$, $0.5$, and $1.0$,
  so that the error (read from top to bottom and then left to right) incurred by this replacement is about $-3{\Ss\rm E}-3$, $-7{\Ss\rm E}-4$,
  $-3{\Ss\rm E}-4$, $-2{\Ss\rm E}-2$, $-2{\Ss\rm E}-2$, $-1.3{\Ss\rm E}-1$, and $-2{\Ss\rm E}-4$.
\end{footnotesize}
\end{table}
\section{Ramifications}\label{conseqsec}\req
\subsection{Ideal distributions with polynomially decaying tails} \label{fattail}
In order to be able to cover ideal distributions with polynomially decaying tails,
we sharpen the restriction
of the original neighborhood system $\tilde{\cal Q}_{n}(r,\ve_0)$ from
\eqref{modifdefi} to
\begin{equation}\label{modifdefi1}
Q_{n}={\cal L}\Big\{[(1-U_i)X_i^{\SSs \rm id}+U_i X_i^{\SSs \rm di}]_i\,\Big|\,\limsup_{n}\Tfrac{1}{n}\sum_{i=1}^n U_i\leq  \ve_0'\,\Big\}
\end{equation}
for some fixed $\ve_0'$ such that
\begin{equation} \label{s0def}
0\le \ve_0'< \ve_0
\end{equation} giving the new neighborhood system $\tilde {\cal Q}'_{n}(r;\ve_0')$.
Correspondingly, we will consider the asymptotics of
\begin{equation}
R_n'(S_n,r;\ve_0'):= \sup_{Q_n\in \tilde{\cal Q}'_n(r;\ve_0')} n\, \int |S_n-\theta_0|^2 \,dQ_n
\end{equation}
It is not surprising that all results up to this point on maximal risks are unaffected by this subtle modification.
But, 
we may replace assumption (Vb) by
\begin{itemize}
  \item[(Pd)]   There are some $T>0$ and $\eta>0$ such that
  \begin{equation}\label{abkling}
F(t)\geq 1-t^{-\eta},\quad \mbox{for } t>T,\qquad F(t)\leq (-t)^{-\eta}\quad\mbox{ for } t<-T
\end{equation}
\end{itemize}
\begin{Prop}\label{abklingP}
In the location model of Subsection~\ref{1dimlocset},
assume {(bmi)}, {(D)}, and {(C)} from section~\ref{Assum};
additionally assume that the central distribution $F$ satisfies \eqref{abkling}.
Then, on $\tilde {\cal Q}'_{n}(r;\ve_0')$, the assertions of Theorem~\ref{main} ---with any $k_2>2$--- continue to hold.
\end{Prop}
Property~\eqref{abkling} can be made plausible by the following proposition:
\begin{Prop}\label{slowdecay}
In the location model of Subsection~\ref{1dimlocset}, assume: For any $d>0$,
\begin{equation}\label{abkl}
\liminf_{t\to\infty} t^{d}(1-F(t))>0 \qquad\mbox{or}\qquad \liminf_{t\to\infty} t^{d}F(-t)>0
\end{equation}
Then for any sample size $n$, the MSE of the M-estimator $S_n$ to any IC $\psi$ according to (bmi) in the ideal model is infinite.
\end{Prop}
%
%
Conditions~\eqref{contbed1} resp.\ \eqref{contbed2} almost characterize the risk-maximizing contaminations:
\begin{Prop} \label{neccond}
Under the assumptions of Theorem~\ref{main}, let $\delta_0,c_0>0$. Assume that $\hat b=b$
and let $B_n:=\inf\{x\,\big|\, \psi(x)\ge b - c_0/\sqrt{n}\}$. Assume that, for $K=\sum_{i=1}^n U_i$ and $k>(1-\delta)r\sqrt{n}$,
\begin{equation} \label{condnec}
\Pr\Big(\sum_{i=1}^n U_i \Jc(X_i^{\rm \SSs di}\le B_n+v_0\sqrt{\log(n)/n})\geq 1\,\Big|K=k\Big)\ge p_0>0
\end{equation}
Then, eventually in $n$, for any such sequence of contaminations $Q^\flat_n\in\tilde{\cal Q}(r)$, the maximal MSE as in condition~\eqref{contbed2}
(i.e.\ with positive bias) in \eqref{mainres}  cannot be attained.
More precisely,
\begin{equation}
R_n(S_n,r)-n \Ew_{Q^\flat_n}S_n^2\ge
{2p_0v_0}(rc_0+b)/({n\sqrt{2\pi}})
\end{equation}
A corresponding relation holds for condition~\eqref{contbed1}.
\end{Prop}

\subsection{Convergence of variance and bias separately}
The technique used to derive  Theorem~\ref{main} also applies if we are interested in
variance and bias separately; we get
\begin{Prop}\label{biasvar}
  Under Assumptions~{(bmi)} to {(C)} and for  sample size $n$,
  an M-estimator $S_n$ for scores-function $\psi$
  under a  measure $Q_n^0\in\tilde {\cal Q}_n(r;\ve_0)$ according to \eqref{contbed1} resp.\ \eqref{contbed2} 
  admits the following expansions
\begin{eqnarray}
\sqrt{n}\,\Big|{\rm Bias}(S_n,Q_n^0)\,\Big|&=&
\Big|\,
rb+\Tfrac{1}{\sqrt{n}\,}\,B_{1,0}+\Tfrac{r^2}{\sqrt{n}\,}\,B_{1,1}+\Tfrac{r}{n}\,B_2\,\Big|
+\Lo(n^{-1})\label{bias}\\
n\,{\rm Bias}^2(S_n,Q_n^0)&=&
r^2b^2+\Tfrac{r}{\sqrt{n}\,}\,C_1+\Tfrac{1}{n}\,C_2
+\Lo(n^{-1})\label{bias2}\\
n\,{\rm Var}(S_n,Q_n^0)&=&v_0^2+\Tfrac{r}{\sqrt{n\,}}\,D_1+\Tfrac{1}{n}\,D_2
+\Lo(n^{-1})\label{var}
\end{eqnarray}
with
\begin{small}
\begin{eqnarray}
B_{1,0}&=&(\Tfrac{1}{2}l_2+\tilde v_1)v_0^2,\qquad
B_{1,1}\;=\;\;b( 1\pm\Tfrac{1}{2}l_2b)\\
B_2&=&\Big[(\Tfrac{1}{2}{l_2^2}+\Tfrac{1}{6}l_3 ){b}^{3}+ b \pm l_2{b}^{2}\Big]{r}^{2}+b(1\pm \Tfrac{1}{2}l_2b)+\nonumber\\
&&\quad+\Big[ (\Tfrac{1}{2}l_3+\Tfrac{3}{2}\,{l_2^2}+\tilde v_2+{\tilde v_1}^{2}+3\,\tilde v_1\,l_2)b\pm
   \Tfrac{1}{2}l_2\pm\tilde v_1\Big]{v_0}^{2}
\\[0.5ex]
C_1&=&b^2r^2(\pm l_2b+2)\pm b(l_2+2\tilde v_1)v_0^2\\
C_2&=&  ({\tilde v_1}\,l_2+\Tfrac{1}{4}\,{l_2^2}+{\tilde v_1}^{2}){v_0}^{4}+
\Big[3\,{b}^{2}\pm 3\,l_2\,{b}^{3}+
   (\Tfrac{5}{4}\,{l_2^2}+\Tfrac{1}{3}\,l_3){b}^{4}\Big]{r}^{4}+\nonumber \\
   &&+
   \Big[(\Tfrac{7}{2}\,{l_2^2}+l_3+2\,\tilde v_2+2\,{\tilde v_1}^{2}+
   7\,\tilde v_1\,l_2){b}^{2}\,{v_0}^{2}\pm (2\,l_2+4\,\tilde v_1)\,b\,v_0^2+  2{b}^{2}\pm
    l_2\,{b}^{3}\Big]\,{r}^{2}\\
D_1&=&\Big [\pm 2(\,l_2+\,\tilde v_1)b+1\Big]{v_0}^{2}+{b}^{2}\\
D_2&=&
(l_3+\Tfrac{7}{2}\,{l_2^2}+11\,\tilde v_1\,l_2+8\,{\tilde v_1}^{2}+
3\,\tilde v_2){v_0}^{4}+\Big (\Tfrac{2}{3}\,{\rho_1}+
(l_2+2\,\tilde v_1 ){\rho_0}\Big ){v_0}^{3}+
\nonumber \\
&&\Big[\Big((l_3+{\tilde v_1}^{2}+
\tilde v_2+5\,\tilde v_1\,l_2+4\,{l_2^2} ){b}^{2}\pm 4 (l_2+\tilde v_1)b+1
\Big ){v_0}^{2}\pm 2\,l_2\,{b}^{3}+3{b}^{2}\Big]{r}^{2}
\end{eqnarray}
\end{small}

where we are in the $-\,[+]$-case according to whether \eqref{contbed1} or \eqref{contbed2} applies.
\end{Prop}
For a proof to this proposition, we may proceed exactly as in the proof of Theorem~\ref{main};
only in \eqref{intds}, we keep the integration domain and replace the
integrand $u_1(s)^2 \,\varphi(s)\,g_n(s)$ by $ u_1(s) \,\varphi(s)\,g_n(s)$;
we do not spell this out here. In {\tt MAPLE} the expressions are obtained by means
of our procedure {\tt asESi}.
%

\appendix
\section{Proofs}\label{proofsec}\req
%
\subsection{Proof to Lemma~\ref{clemn}}\label{clemnsec}
Let $G_t$ be the law of $\psi_t(X^{\SSs\rm id})$.
By assumption, the Lebesgue decomposition yields $dG_0= a g\,d\lambda + (1-a)\,d\tilde G$
for $a\in(0,1]$, $g$ some probability density and $\tilde G \perp \lambda$. The support of
$g$ contains an open interval $(c_1,c_2)$ and $G_0(c_2)>G_0(c_1)$.
On $(c_1,c_2)$, $\psi$ is strictly isotone and continuous,
so that with $d_i=\psi^{-1}(c_i)$
\begin{equation}
P(\psi_t(X^{\SSs\rm id})\in(c_1,c_2))=P(d_1+t < X^{\SSs\rm id}< t+d_2)
=\int^{d_2+t}_{d_1+t}\,dF
\end{equation}
But \begin{equation}
\int^{d_2+t}_{d_1+t}\,dF=G_0(c_2)-G_0(c_1)+\Lo(t^0)
\end{equation}
so that for $t$ small enough, the absolute continuous part of $G_t$ is uniformly bounded away from $0$ and hence
by the Lebesgue Lemma our condition~\eqref{(C)gl} holds.
\hfill\qed
\begin{small}
\subsection{Proof to Proposition~\ref{normalspclem}}\label{pnormalspclem}
To get $\Ew[\hat \eta_c\Lambda_f]=1$, the Lagrange multiplier $A_c$ must be determined by
 $ A_c^{-1}=2\Phi(c)-1 \nonumber$. 
It holds that $b=A_cc$. For $c\to \infty$ we obtain the classically optimal IC, and $c\to 0$, using l'Hospital  yields the IC of the
sample median. As to $L(t)$,  we obtain
\begin{equation}
  L_c(t)=A[c-(c+t)\Phi(t+c)+(t-c)\Phi(t-c)+\varphi(t-c)-\varphi(t+c)],\quad 
  L_{\infty}(t)=-t,\quad L_0(t)=\sqrt{\Tfrac{\pi}{2}}\,(1-2\Phi(t))
\end{equation}
all arbitrarily often differentiable functions, so the $l_i$-part of {(D)} holds as stated in the proposition. 
For $V(t)$ introduce
\begin{equation}
S(t):=\Ew[\psi(x-t)^2],\qquad W(t):=V(t)^2\nonumber
\end{equation}
Then, suppressing the argument $t$,
$W=S-L^2,\qquad W'=S'-2LL',\qquad W''=S''-2L'^2-2LL'' \nonumber
$
and with
$W_0=W(0)$, 
$\tilde W_1(0)=W'(0)/W_0$, 
$\tilde W_2(0)=W''(0)/W_0$,  
we get
\begin{equation}
W_0=S(0),\qquad \tilde W_1 = S'(0)/S(0),\qquad\tilde W_2 = (S''(0)-2)/S(0)\nonumber
\end{equation}
and hence
$V(t)
=\sqrt{W_0}(1+\Tfrac{\tilde W_1\,t}{2}+\Tfrac{(2\tilde W_2-\tilde W_1^2)\,t^2}{8})+\LO(t^{2+\delta})\nonumber$
so that
\begin{equation}
v_0=\sqrt{S(0)},\qquad \tilde v_1=\frac{S'(0)}{2S(0)},\qquad \tilde v_2=\frac{2S''(0)-4-S'(0)^2/S(0)}{4S(0)}\nonumber
\end{equation}
In our case we have for $0<c<\infty$
\begin{eqnarray}
S(t)&=&A^2_c\Big[c^2\big(1-\Phi(t+c)+\Phi(t-c)\,\big)+(1+t^2)\big(\Phi(t+c)-\Phi(t-c)\,\big)
+(t-c)\varphi(t+c)-(t+c)\varphi(t-c)\,\Big]\nonumber
\end{eqnarray}
and
$  S(t)=1+t^2\quad\mbox{for }c=\infty, \qquad S(t)=\frac{\pi}{2}=b^2 \quad\mbox{for }c=0\nonumber$,
so \eqref{VD} holds with\newline
\centerline{
\begin{tabular}{r|c|c|c}
&$0<c<\infty$&$c=0$&$c=\infty$\\
\hline
$S(0)$&$2 b^2(1-\Phi(c))+A_c(1-2b\varphi(c))$ &$1$&$\Tfrac{\pi}{2}$\\
$S'(0)$&$0$&$0$&$0$\\
$S''(0)$&$2A_c^2(2\Phi(c)-1-2c\varphi(c))$ &$2$&$0$\\[1ex]
\end{tabular}\vspace{1ex}}\newline
and the assertions as to $v_0$, $\tilde v_1$, $\tilde v_2$ 
follow.
As to (Vb), for $|t|\to \infty$,
we get with Mill's ratio for any $\delta>0$
$$
\Big|\,b-|L(t)|\,\Big|
=A_c \Big|\,(c+t)\bar \Phi(t+c)-(t-c)\bar \Phi(t-c)+\varphi(t-c)-\varphi(t+c)\,\Big|=\nonumber\\
=\Lo(\exp(-\frac{t^2}{2+\delta}))
$$
Again with Mill's ratio,
$|S(t)-b^2|\le
A_c^2\Big[2(t^2+1)\bar\Phi(|t|-c)+2(|t|+c)\varphi(|t|-c)\,\Big]=
\Lo(\exp(-\frac{t^2}{2+\delta}))
$ 
and hence
$V^2(t)=S(t)-L(t)^2=\Lo(\exp(-\frac{t^2}{2+\delta})) 
$. 
For $c =0$ we get
$\Big|\,b-|L(t)|\,\Big|=\sqrt{2\,\pi}\, \bar \Phi(t)
=\Lo(\exp(-t^2/2))$ and 
\begin{equation}
V^2(t)=b^2-(b+\Lo(\exp(-t^2/2)))^2=\Lo(\exp(-t^2/2))\nonumber
\end{equation}
For $\rho(t)$ and $\kappa(t)$, we introduce
$  M(t):=\Ew[\psi(X-t)^3]$, $N(t):=\Ew[\psi(X-t)^4]\nonumber $.
Then, again suppressing the argument $t$
\begin{equation}
\rho=V^{-3}[M-3LS+2L^3],\qquad\kappa=V^{-4}[N-4ML+6SL^2-3L^4]-3\nonumber
\end{equation}
and hence
$
\rho_0=v_0^{-3}M(0)$, 
$\kappa_0=V^{-4}N(0)-3$. 
For $\rho_1$ we note
\begin{equation}
\rho'=V^{-3}\Big(-3[M-3LS+2L^3] \,V'/V+(M'-3L'S-3LS'+3L'L^2)\Big)\nonumber
\end{equation}
so that
$\rho_1=v_0^{-3}(-3M(0)\tilde v_1+M'(0)+3S(0))$. 
In our case, for $c=\infty$,
$
M(t)=-3t-t^3$, $M'(t)=-3-3t^2$, 
$N(t)=t^4+6t^2+3\nonumber
$
and for $c=0$,
$  M(t)=(\sqrt{\frac{\pi}{2}}\,)^3(1-2\Phi(t))$, 
  $M'(t)=-2\,(\sqrt{\frac{\pi}{2}}\,)^3\varphi(t)$, 
  $N(t)=\frac{\pi^2}{4}$, while
for $0<c<\infty$
\begin{eqnarray}
  M(t)&=&A_c^3\Big[c^3-\Phi(t+c)(c^3+t^3+3t)-\Phi(t-c)(c^3-t^3-3t)+\nonumber\\
  &&\qquad+(t^2+tc+2+c^2)\varphi(t-c)-(t^2-tc+c^2+2)\varphi(t+c)\Big]\nonumber\\
  M'(t)&=&A_c^3\Big[3\big(\Phi(t-c)-\Phi(t+c)\big)(t^2+1)-
  3(t-c)\varphi(t+c)+3(t+c)\varphi(t-c)\Big]\nonumber\\
  N(t)&=&A_c^4\Big[c^4+\big(\Phi(t+c)-\Phi(t-c)\big)(t^4+6t^2+3-c^4)+
  (t^3-t^2c+tc^2-c^3+5t-3c)\varphi(t+c)-\nonumber\\
  &&\qquad-(t^3+t^2c+tc^2+c^3+5t+3c)\varphi(t-c)\Big]\nonumber
\end{eqnarray}
This gives the assertion as to $\rho_0$, $\rho_1$ and $\kappa_0$, and hence \eqref{KD} holds.\\
For $c>0$, $\Pr(|\eta_c|<b)>0$ and $\eta_c$ is continuous.
But, on $\{|\eta_c|<b\}$, ${\cal L}(\eta_c)$ is a.c.\ and hence Lemma~\ref{clemn} entails (C).
\hfill\qed
\end{small}
\subsection{Proof of Theorem~\ref{main}}\label{prmain}
We plug in $(X_i)\sim Q_n$ for some $Q_n \in \tilde{\cal Q}_n(r)$
into the defining relations for M-estimators of \eqref{mest1}.
\paragraph{\small Outline of the proof}
We begin with conditioning w.r.t.\ the number $K=\sum_i U_i=k$ of contaminated observations;
next for fixed $t\in\R$, we consider $\tilde T_{n,k,t}(t)=\sum_{i:U_i=1}\psi(X_i-t)$ and
condition the probability w.r.t.\ its realization $\tilde t_{n,k,t}$. In the sequel we suppress the indices of $\tilde t_{n,k,t}$.
Denote this event by
\begin{equation}
D_{k,\tilde t}:= \{ K=k,\tilde T_{n,k}(\sqrt{t}\,)=\tilde t\,\}
\end{equation}
Thus
\begin{equation}
n\,{\rm MSE}(S_n,Q_n\,|\,D_{k,\tilde t}\,)=\int_0^{\infty}\!\!\!\!\Pr(S_n^2\geq t\,|\,D_{k,\tilde t}\,)\,dt=
\int_0^{\infty}\!\!\!\!\Pr(S_n\geq \sqrt{t}\,|\,D_{k,\tilde t}\,)\,dt+\int_0^{\infty}\!\!\!\!\Pr(S_n\leq -\sqrt{t}\,|\,D_{k,\tilde t}\,)\,dt\label{738}
\end{equation}
For the sequel, we define
$\bar n:= n-k$, $s_{n,k}:=s_{n,k}(t)=\frac{-\tilde t-\bar n L(t)}{\sqrt{\bar n}\,\,V(t)}$.
To derive the result, we then partition the integrand according to the following tableau 
where $C'>0$ is some constant and $\delta$ is
the exponent from assumption~(Vb): \label{tabelAUF}
\begin{center}
\begin{tabular}{c||c|c|c}
&$K< k_1r\sqrt{n}$ &  $k_1r\sqrt{n}\leq K< \ve_0 n $ & $K\ge  \ve_0 n$\\[0.5ex]
\hline\hline
$|t|\leq k_2b^2{\log(n)/n}$& (I) &&\\[0.8ex]
\cline{1-2}
$k_2 b^2{\log(n)/n} < |t|\leq C n^{1+3/\delta} $& (III)& \raisebox{1.5ex}[-1.5ex]{(II)}& {\footnotesize excluded} \\[0.8ex]
\cline{1-3}
$|t|> C n^{1+3/\delta} $& \multicolumn{2}{c|}{(IV)}& \\[0.5ex]
\hline
\end{tabular}
\end{center}
\noindent
At this point we also summarize the constants that will be used throughout this section.
\begin{center}
\begin{tabular}{c||c|c}
constant &$k_1$&$k_2$\\
\hline value& $>1$&$>2\vee (\frac{3}{2}+\frac{3}{2\delta})$
\end{tabular}
\end{center}
For all cases except for (I), we will show that they contribute only terms of order $\Lo(n^{-1})$ to $n\,{\rm MSE}({S}_n)$
and hence can be neglected. Applying Taylor expansions at large, we derive an expression in which it becomes clear,
that independently from $t$ and eventually in $n$, the maximal MSE is attained for $\tilde t_{n,k}$ either
$kb$ or identically $-kb$ for all $t$ in (I) --- or equivalently all contaminated observations
are either smaller than $\check y_n-k_2b^2\log(n)/n$ or larger than $\hat y_n+k_2b^2\log(n)/n$.
Integrating out first $t$ and then $k$ we obtain the result \eqref{mainres} stated in Theorem~\ref{main}.
\paragraph{\small Conditioning w.r.t.\ the number of contaminated observations}
As announced, for the moment we condition w.r.t. the number $K=\sum_i U_i=k$ of contaminated observations in the sample.
Denote the ideally distributed part as $T_{n,k}(t):=\sum_{i:U_i=0} \psi_t(X_i)$. Then we get
\begin{eqnarray}
\Pr\{S_n\leq t\,\Big |\,K=k\,\}+R_n^{(0)}(k)=\Pr(T_{n,k}(t)<-\tilde T_{n,k}(t))=
\Pr(\frac{T_{n,k}(t)-\bar n L(t)}{\sqrt{\bar n} V(t)}
<-\frac{\tilde T_{n,k}(t)-\bar n L(t)}{\sqrt{\bar n} V(t)}\,)
\end{eqnarray}
where $R_n^{(0)}(k)\not=0$ can only happen for mass points of $\Lw(T_{n,k}(t)+\tilde T_{n,k}(t))$.
\paragraph{\small Conditioning w.r.t.\ the actual contamination}
Next, we condition the probability w.r.t.\ the actual value
of the contamination $\tilde T_{n,k}=\tilde t$. This gives
\begin{equation}
\Pr\{S_n\leq t\, |D_{k,\tilde t}\}+\tilde R_n^{(0)}(k,\tilde t)=
\Pr\Big(\frac{T_{n,k}(t)-\bar n L(t)}{\sqrt{\bar n} \,V(t)}<%
s_{n,k}(t)\Big)\label{740}
\end{equation}
where again $\tilde R_n^{(0)}(k,\tilde t)\not=0$ can only happen for mass points  of $\Lw(T_{n,k}(t))$.
%
%
\paragraph{\small Negligibility of case (IV)}
Without loss, assume that $b=\hat b$. By monotonicity and boundedness in assumption (bmi), to given $0<\eta<-\check b$ there is a $t_0>0$
such that for $t>t_0$, $$\check b<L(t)=\Ew[\psi(X^{\rm \SSs id}-t)]\le \check b +\eta$$ Let $t_1>t_0$, $\delta>0$ and $C'>0$ so that for
$t>t_1$, by (Vb), $|V(t)|\leq C' t^{-1-\delta}$. Then we apply the Chebyshev inequality to obtain for $t>t_1^2$
\begin{eqnarray}
&&\Pr\{S_n> \sqrt{t} \,\Big|\,D_{k,\tilde t}\}
\stackrel{}{\leq}\Pr\Big(T_{n,k}(\sqrt{t}\,)-\bar n L(\sqrt{t}\,)\ge -\tilde t-\bar n L(\sqrt{t}\,)\Big)
\stackrel{\mbox{\tiny Cheb.}}{\leq}\frac{\bar nV^2(\sqrt{t}\,)}{(\tilde t+\bar n L(\sqrt{t}\,))^2}
\stackrel{{\rm\SSs (Vb)}}{\leq}\frac{\bar n C' t^{-(1+\delta)}}{(\tilde t+\bar n L(\sqrt{t}\,))^2}\leq\nonumber\\
&\le&\frac{n C' t^{-(1+\delta)}}{(\tilde t+\bar n \check b +\eta
)^2}
\stackrel{\SSs \tilde t\leq k\hat b}{\leq}
\frac{n C' t^{-(1+\delta)}}{[k \hat b +\bar n\check b+\eta]^2}
=\frac{n C' t^{-(1+\delta)}}{[k (\hat b-\check b) +n\check b+\eta]^2}
\stackrel{{\SSs k\leq \ve_0 n}}{\leq}\frac{n C' t^{-(1+\delta)}}{(\check b-\eta)^2}
\end{eqnarray}
and correspondingly (with $b=-\check b)$ for $\Pr\{S_n\leq -\sqrt{t} \,\big|\,D_{k,\tilde t}\}$;
but
\begin{equation}
\frac{C' n^2 }{ (b-\eta)^2} \int_{Cn^{1+3/\delta}}^{\infty} t^{-(1+\delta)}\,dt= \frac{C'  C^{-\delta} n^{-1-\delta}}{\delta (\check b-\eta)^2}=\Lo(n^{-1})
\end{equation}
\paragraph{\small Negligibility of case (II)}
\begin{Lem}\label{binlem}
Let
\begin{equation}
\kappa:=k_1\log k_1+1-k_1
\end{equation}
Then it holds that
\begin{equation}\label{alphahoeff}
\Pr({\rm Bin}(n,r/\sqrt{n\,}\,)>k_1 r\sqrt{n})\leq\exp\big(-\kappa\, r\sqrt{n}+\Lo(\sqrt{n}\,)\big)
\end{equation}
\end{Lem}
\begin{proof}{} \citet[Lem.~A.2]{Ruck:03b}
\hfill\qed\end{proof}

As in (II), $|t|<C n^{1+3/\delta}$, the integrand of $n \, {\rm MSE}(S_n,Q_n \,|\,\,D_{k,\tilde t}\,)$ is bounded by some polynomial in $n$,
and hence by Lemma~\ref{binlem} the contribution of (II) is indeed $\Lo(n^{-1})$.

Another consequence of the exponential decay of \eqref{alphahoeff} is
that we may neglect values of $K>k_1(n) r\sqrt{n}$  when integrating along $K$.
\begin{Cor} \label{corewk}
Let $K\sim{\rm Bin}(n,r/\sqrt{n}\,)$. Then, in the setup of Lemma~\ref{binlem}, for any $j\in\N$,
\begin{equation}
\Ew[K^j \Jc_{\{X\geq k_1(n) r \sqrt{n}\}}]=\Lo(e^{-rn^{d}})
\end{equation}
for any $0<d<\sqrt{n}$.
\end{Cor}
\begin{proof}{}
$\Ew[K^j \Jc_{\{K\geq k_1(n) r \sqrt{n}\}}]
\leq n^j\Pr(X>k_1(n)r\sqrt{n})\stackrel{\eqref{alphahoeff}}{=} \Lo(e^{-rn^{d}})
$.
\hfill\qed\end{proof}
\paragraph{\small Negligibility of case (III)}
We apply Hoeffding's first bound from Lemma~\ref{hoef}:
\begin{equation}\label{abschtz}
\Pr\{S_n> \sqrt{t} \,\Big|\,D_{k,\tilde t}\}\leq \Pr(T_{n,k}(\sqrt{t}\,)\geq -\tilde t\,\big|\,\,D_{k,\tilde t}\,)\leq \exp(-2n\Delta^2/b^2)
\end{equation}
for $\Delta:=-L(\sqrt{t}\,)-\frac{\tilde t}{n}$. As $\psi$ is isotone, $L$ is antitone, hence in case (III),
\begin{equation} \label{8.22}
L(\sqrt{t}\,)\leq  L(b\sqrt{k_2\log(n)/n}\,)=-b\sqrt{k_2\log(n)/n}+\Lo(\sqrt{\log(n)/n}\,)
\end{equation}
Thus
\begin{equation}
  \Delta\geq-L(\sqrt{t}\,)-\frac{kb}{n}
\stackrel{\eqref{8.22}}{>}\frac{b}{\sqrt{n}}[\sqrt{k_2\log(n)}+\Lo(\sqrt{\log(n)}\,)]
\end{equation}
and
$
 \exp(-2\frac{n\, \Delta^2}{b^2})< n^{-2k_2}(1+\Lo(n^0))
$.  
This latter is $\Lo(n^{-3-3/\delta})$ and thus integrating $n\, {\rm MSE}$ out along (III) we get something of order $\Lo(n^{-1})$.
%
\paragraph{\small Asymptotic normality}
On (I), by Lemma~\ref{Kohllem}
\begin{equation}
\Pr\big\{S_n\geq \sqrt{t} \,\Big|\,D_{k,\tilde t}\,\big\}=
\Pr\left(\frac{T_{n,k}(\sqrt{t}\,)-\bar n L(\sqrt{t}\,)}{\sqrt{\bar n}\,V(\sqrt{t}\,)}>s_{n,k}(t)\right)+\LO(e^{-\gamma n})
\end{equation}
for some $\gamma>0$, uniformly in $t$ and $k$.
For $i=1,\ldots,\bar n$, let $j_i\in\{1,\ldots,n\}$ be the indices such that $U_{j_i}=0$.
We may apply Theorem~\ref{berry2}(b) to \eqref{738}/\eqref{740},
identifying 
\begin{equation}
\xi_{i,t}:=\frac{1}{ V(t)}[\psi_t(X_{j_i})-L(t)],\qquad i=1,\ldots,\bar n
\end{equation}
and setting $\Theta:=\Theta_n=\{|t|\leq k_2b^2\log(n)/n\}$.
This application is possible, as $|\psi|<b$, so $\sup_{t\in\Theta_n} \Ew|\tilde \xi_{i,t}|^5<\infty$. By condtion~(C)
of our assumptions, Cram\'er condition \eqref{Cddd} of the theorem holds if $n$ is large enough.\\
We note that if in Theorem~\ref{main}, we limit ourselves to term  $A_1$ and hence only assume (C'), we may
apply Theorem~\ref{berry2}(a).
\\
With $G_{n,t}(s)$ from \eqref{Gnstdef} we define
$\tilde G_{n,t}(u):=G_{n,t}(s_{n,k}(u))$, $\tilde G_{n}(t):=\tilde G_{n,t}(t)$
and obtain for $|t|\leq k_2 b^2 \log(n)/n$ and $K<k_1r\sqrt{n}$ uniformly in $t$ and $k$:
\begin{eqnarray}
&&\LO(\exp(-\gamma n))+\Pr\big\{S_n\geq \sqrt{t} \,\Big|\,D_{k,\tilde t}\,\big\}=
\Pr\big(\sum_{i=1}^{\bar n}\xi_{i,\sqrt{t}}>s_{n,k}(\sqrt{t}\,)\big)=1-\tilde G_n(\sqrt{t}\,)+\LO(n^{-3/2})
\label{thm73}
\end{eqnarray}
Hence, using negligibility of (II), (III) and (IV),
and setting
\begin{equation}
n^\natural=\sqrt{\bar n/n},\qquad l_n=n^\natural \sqrt{k_2\log(n) },\qquad l^{(0)}_n=k_2b^2 \log(n)/n
\end{equation}
we obtain
\begin{eqnarray}
&&n\,{\rm MSE}(S_n,Q_n\,\big|\,D_{k,\tilde t}\,)=
(n^\natural)^{-2}\, \bar n\,\int_0^{l_n^{(0)}}\mbox{\hspace{-.2cm}}
1-\tilde G_n(\sqrt{t}\,)+\tilde G_n(-\sqrt{t}\,)\,dt+ \Lo(n^{-1})=\nonumber\\
&=&2(n^\natural)^{-2} \int_0^{bl_n}u\Big(1-\tilde G_n(\frac{u}{\sqrt{\bar n}}\,)+ \tilde G_n(-\frac{u}{\sqrt{\bar n}}\,)\,\Big)\,du+ \Lo(n^{-1})
\label{letzteqn}
\end{eqnarray}
As $\tilde G_n$ is arbitrarily smooth, integration by parts is available and gives
\begin{equation}
n\,{\rm MSE}(S_n,Q_n\,\big|\,D_{k,\tilde t}\,)=R_n+(n^\natural)^{-2}  \int_{-bl_n}^{bl_n}\frac{u^2}{\sqrt{\bar n}}\,\,G_n'(\frac{u}{\sqrt{\bar n}}\,)\,du+ \Lo(n^{-1})
\end{equation}
with %
\begin{equation}
R_n:=k_2\, \log(n)\,b^2\, \big[1-\tilde G_n(b\, \sqrt{\Tfrac{k_2\log(n)}{n}}\,)-\tilde G_n(-b\, \sqrt{\Tfrac{k_2\log(n)}{n}}\,)\,\big]
\end{equation}
A closer look at $s_{n,k}(\pm b\, \sqrt{\Tfrac{k_2\log(n)}{n}}\,)$ reveals
\begin{eqnarray}
s_{n,k}(\pm b\, \sqrt{\Tfrac{k_2\log(n)}{n}}\,)
&
\stackrel{{\rm \SSs \eqref{VD}}}{=}&
\frac{\LO(\sqrt{n}\,)\pm b\, \sqrt{\frac{k_2\bar n^2 \log(n)}{n}}+\LO(\frac{\bar n\,\log(n)}{n}\,)}
{\sqrt{\bar n}\,(v_0+\Lo(n^0))}=
\frac{\pm b \sqrt{k_2\log(n)} }{v_0}(1+\Lo(n^0))
\end{eqnarray}
We also note that, again  by (bmi) $v_0^2= \Ew[\psi^2]\leq b^2$, hence $b/v_0>1$.
In particular, eventually in $n$,
\begin{equation} \label{randgr}
|\tilde s_{n,k}(\pm b \sqrt{k_2\log(n)}\,)|>\sqrt{2\log(n)}
\end{equation}
But, as $|\psi|\leq b$ by (bmi), $|\kappa|\leq b^4$ and $|\rho|\leq b^3$, and thus by Mill's ratio, there is
some $0<K<\infty$, independent of $t$, $n$, such that for any $s>0$
\begin{equation} \label{mill2}
\max\big(1-G_{n,t}(s),\,G_{n,t}(-s)\big) \leq K |s|^5 \,\exp(-s^2/2)
\end{equation}
Thus  for $n$ sufficiently large
\begin{equation}
1-\tilde G_n(b\, \sqrt{\Tfrac{k_2\log(n)}{n}}\,)=\exp(-\frac{k_2b^2\log(n) }{2v_0^2}+\Lo(n^0)))
=  \LO( \frac{\log(n)^{5/2}}{n^{1+\delta}})
\end{equation}
for some $\delta>0$. The same goes for $\tilde G_n(-2b\, \sqrt{\Tfrac{\log(n)}{n}}\,)$, and therefore,
$
R_n=\LO(\log(n)^{7/2}/n^{1+\delta})=\Lo(n^{-1})
$ 
and
\begin{equation}
n\,{\rm MSE}(S_n,Q_n\,\big|\,D_{k,\tilde t}\,)=(n^\natural)^{-2}\int_{-bl_n}^{bl_n}\frac{u^2}{\sqrt{\bar n}}\,\,
G_n'(\frac{u}{\sqrt{\bar n}}\,)\,du+ \Lo(n^{-1})\label{letzteq}
\end{equation}
To make more transparent, which terms are bounded to which degree, we introduce
the following notation, which will also help {\tt MAPLE}
to ignore irrelevant terms
${t^{\natural}}:=\Tfrac{\tilde t}{\sqrt{ \bar n}}$,
$\tilde s_{n,k}(x)=s_{n,k}(\Tfrac{x}{\sqrt{\bar n}}\,)$.
Then on (I), $u=\LO(\sqrt{\log (n)}\,)$, ${t^{\natural}}=\LO(n^0)$.
In particular this will not affect the remainder terms of the Taylor expansions of
assumption (D).\\
In the sequel, we drop the indices of $s_{n,k}$ and $\tilde s_{n,k}$, where they are clear from the context.
Next, we spell out $\tilde G'_{n}(u)$ in \eqref{letzteq} more explicitly.
Denote
\begin{equation}
{\cal G}_n(s,t):=G_{n,t}(s),\quad G^{(1)}_{n,t}(s):=[\Tfrac{\partial}{\partial s}{\cal G}_n] (s,t),\quad G^{(2)}_{n,t}(s):=
[\Tfrac{\partial}{\partial t}{\cal G}_n] (s,t)
\end{equation}
Then, as $
\tilde s_{n,k}'(x)=s_{n,k}'(\frac{x}{\sqrt{\bar n}}\,)/{\sqrt{\bar n}},
$
$$
\hspace{-1em}\tilde G'_{n}(\Tfrac{u}{\sqrt n}\,)=
[G^{(1)}_{n,x}(s(x))s'(x)+G^{(2)}_{n,x}(s(x))]\,\Big|_{x=\frac{u}{\sqrt{\bar n}}}=
G^{(1)}_{n,u/{\sqrt{\bar n}}}(\tilde s(u))\,\tilde s'(u) \sqrt{\bar n}+G^{(2)}_{n,u/{\sqrt{\bar n}}}(\tilde s(u))=:
\tilde g_n(u) \sqrt{\bar n}
$$ 
and therefore
\begin{equation}
n\,{\rm MSE}(S_n,Q_n\,\big|\,D_{k,\tilde t}\,)=(n^\natural)^{-2}\int_{-bl_n}^{b l_n} \, u^2 \tilde g_{n}(u)\,du+\Lo(n^{-1})\label{letz2eq}
\end{equation}
%
\paragraph{\small Expanding $\tilde g_{n}(u)$}\label{neuexp}
Considering $\tilde g_{n}(u)$ more closely, we expand the terms according to assumption (D)
---with the help of our {\tt MAPLE} procedures {\tt asS}, {\tt asS1}, {\tt asg}
\begin{eqnarray}
\tilde s(u)&=&\frac{-t^\natural -\sqrt{\bar n} L(\Tfrac{u}{\sqrt{\bar n}})}{V(\Tfrac{u}{\sqrt{\bar n}})}= 
\Tfrac{1}{v_0}\Big[(u-t^\natural)- \Tfrac{u}{\sqrt{\bar n}}\big(\Tfrac{l_2u}{2}+\tilde v_1(u-t^\natural)\big)+\nonumber\\
&&\quad +\Tfrac{1}{\bar n}\Big((l_2\Tfrac{\tilde v_1}{2}-\Tfrac{l_3}{6}) u^3+(u-t^\natural)u^2(\tilde v_1^2-\tilde v_2/2)\Big)\Big]+\LO(n^{-(1+\delta)})\\
\tilde s'(u)&=&-\frac{L'(\Tfrac{u}{\sqrt{\bar n}})}{V(\Tfrac{u}{\sqrt{\bar n}})}+\frac{(t^\natural+L(\Tfrac{u}{\sqrt{\bar n}}))V'(\Tfrac{u}{\sqrt{\bar n}})}{V^2(\Tfrac{u}{\sqrt{\bar n}})}=
\Tfrac{1}{v_0}\Big[1-l_2\Tfrac{u}{\sqrt{\bar n}}-2\tilde v_1 \Tfrac{u}{\sqrt{\bar n}}+\Tfrac{t^\natural}{\sqrt{\bar n}}\tilde v_1+\nonumber\\
&&\quad + \Tfrac{1}{\bar n}\Big((3\tilde v_1^2-\Tfrac{l_3}{2}-\Tfrac{3}{2}\tilde v_2+\Tfrac{3}{2}\tilde v_1l_2) u^2 + u t^\natural(\tilde v_2-2\tilde v_1^2)\Big)\Big]+\LO(n^{-(1+\delta)})
\end{eqnarray}
as well as
\begin{eqnarray}
G^{(1)}_{n,u/{\sqrt{\bar n}}}(\tilde s)
&=&\varphi(\tilde s)
\Big[1+\Tfrac{1}{6\sqrt{\bar n}}(\rho_0+\rho_1\Tfrac{u}{\sqrt{\bar n}})\,{(\tilde s^3-3\tilde s)}
+\Tfrac{1}{24n}\kappa_0{(\tilde s^4-6\tilde s^2+3)}+\nonumber\\
  &&\qquad+\Tfrac{1}{72n}\rho_0^2{(\tilde s^6-15\tilde s^4+45\tilde s^2-15)}\Big]+\LO(n^{-(1+\delta)})
\end{eqnarray}
and respectively,
$
G^{(2)}_{n,u/{\sqrt{\bar n}}}(\tilde s)
=\varphi(\tilde s)\Tfrac{\rho_1}{6\sqrt{\bar n}}(1-\tilde s^2)+\LO(n^{-(1/2+\delta)})
$. 
This gives
\begin{equation}
\tilde g_{n}(u)=v_0 \varphi(\tilde s)[1+\Tfrac{1}{\sqrt{\bar n}} P_1(u,t^\natural)+\Tfrac{1}{\bar n} P_2(u,t^\natural)]+\LO(n^{-(1+\delta)})
\end{equation}
for
\begin{equation}
P_1(u,t^\natural)=-l_2 u - 2 \tilde v_1 u + t^\natural \tilde v_1 + \Tfrac{\rho_0}{6v_0^3}(u-t^\natural)^3-\Tfrac{\rho_0}{2v_0}(u-t^\natural)\\
\end{equation}
and $P_2(u,t^\natural)$ a corresponding polynomial in $u$, $t^\natural$, $\tilde v_1$, $\tilde v_2$, $l_2$, $l_3$, $\rho_0$, $\rho_1$, and $\kappa_0$,
the exact expression of which may be taken from our {\tt MAPLE} procedure {\tt asg}.\\
To be able to calculate the integrals, we expand $\varphi(\tilde s) $ in a Taylor expansion about
$s_1=(u-t^\natural)/v_0$
as
\begin{equation}
\varphi(\tilde s) =\varphi(s_1) [1-s_1(\tilde s-s_1)+(s_1^2-1)(\tilde s-s_1)^2/2]+\LO(n^{-(1+\delta)})
\end{equation}
and hence
$\tilde g_{n}(u)=v_0 \varphi(s_1)g_n(s_1)+\LO(n^{-(1+\delta)})$
with
$
g_n(s_1):=1+\Tfrac{1}{\sqrt{\bar n}} \tilde P_1(s_1,t^\natural)+\Tfrac{1}{\bar n} \tilde P_2(s_1,t^\natural)
$ 
for
\begin{eqnarray}
\tilde P_1(s_1,t^\natural)&=&\rho_0\frac{s_1^3-3s_1}{6}+(\Tfrac{l_2}{2}+\tilde v_1)s_1^3-(l_2+2\tilde v_1)s_1v_0+
(l_2+\tilde v_1)[s_1^2-1]t^\natural+\frac{(t^\natural)^2l_2 s_1}{2v_0}
\end{eqnarray}
and $\tilde P_2(s_1,t^\natural)$ a corresponding polynomial 
again to be looked up from our {\tt MAPLE} procedure {\tt asgns}.
This gives
\begin{equation}
n\,{\rm MSE}(S_n,Q_n\,\big|\,D_{k,\tilde t}\,)=(n^\natural)^{-2}\int_{-bl_n/v_0}^{b l_n/v_0} \, h_n(s)\varphi(s)\,\lambda(ds)
+\Lo(n^{-1})\label{letz3eq}
\end{equation}
for
\begin{equation}\label{intds}
h_n(s)=u_1(s)^2 g_n(s),\qquad u_1(s)=sv_0+t^\natural
\end{equation}
\paragraph{\small Selection of the least favorable contamination} \label{lfselec}
Function $h_n(s)$ from \eqref{intds} is a polynomial in $s$, hence on (I), where $|s|=\LO(\log(n))$, we may ignore terms
of (pointwise-in-$s$) order $\LO(n^{-(1+\delta)})$. This gives a complicated expression of form
\begin{equation}
h_n(s)=(sv_0+t^\natural)^2+\frac{1}{\sqrt{n}}Q_1+\frac{1}{n}Q_2
\end{equation}
where $v_0 Q_1$ is a polynomial in $s$, $t^\natural$, $v_0$, $l_2$, $\tilde v_1$, and $\rho_0$ with $\mathop{\rm deg}(Q_1,s)=5$ and  $\mathop{\rm deg}(Q_1,t)=4$,
and $v_0^2 Q_2$ is a polynomial in $s$, $t^\natural$, $v_0$, $l_2$, $\tilde v_1$, $\rho_0$, $l_3$, $\tilde v_2$, $\tilde \rho_1$, and $\kappa_0$
with $\mathop{\rm deg}(Q_2,s)=8$ and  $\mathop{\rm deg}(Q_1,t)=6$; the exact expressions are available on the web-page and may
be generated by our {\tt MAPLE}-procedure {\tt ashn}. Denoting the second partial derivative w.r.t.\ $t^\natural$ by an index $t,t$
we consider
$
h_{n,t,t}(s)=2+\frac{1}{\sqrt{n}}Q_{1,t,t}+\frac{1}{n}Q_{2,t,t}
$ 
where $\mathop{\rm deg}(Q_{1,t,t},s)=3$ and  $\mathop{\rm deg}(Q_{2,t,t},s)=6$, and under symmetry, more specifically
\begin{equation} \label{symdef}
l_2=\tilde v_1=\rho_0=0
\end{equation}
$Q_{1,t,t}=0$ and  $\mathop{\rm deg}(Q_{2,t,t},s)=4$. That is, on (I), uniformly in $s$,
$h_{n,t,t}(s)=2+\LO(\log(n)^3/\sqrt{n})$, and under \eqref{symdef}, the remainder is even $\LO(\log(n)^4/n)$.
Hence eventually in $n$, uniformly in $s$, $h_n$ is strictly convex in $t^\natural$, hence takes its maximum on the boundary,
that is for $|t^\natural|$ maximal.

Going back to the definition of $t^\natural$, we note that for fixed $n$ and $k$,
$
t^\natural=\tilde t/\sqrt{\bar n}=\sum_{i:U_i=1}\psi(X_i-t)/\sqrt{\bar n}
$. 
Obviously, $\tilde t$ is bounded in absolute value by $kb$. This value may be attained if
(up to $\LO(n^{-1})$) all terms $\psi(X_i-t)$ are either $b$ or $-b$ for all $t$ in (I).
This amounts to concentrating essentially all the contamination either right of $\hat y_n+b\sqrt{k_2\log(n)/n}\,$ or left of
$\check y_n-b\sqrt{k_2\log(n)/n}\,$; the decision which of the two alternatives is least favorable is deferred to subsubsection~\ref{Decalter}.

As we may allow for deviations from this ``outlyingness'' as long as we do no affect the expansion of the MSE up to $\LO(n^{-1})$,
we may weaken the concentration property to 
\eqref{contbed1} resp.\ \eqref{contbed2}:
On (I), $|{t^{\natural}}|$ is bounded, so 
smallness of the probabilities in 
\eqref{contbed1} resp.\ \eqref{contbed2} entails that also the expectations of $(t^\natural)^j$,
$j=1,\ldots,6$ arising in $h_n(s)$ are $\Lo(n^{-1})$.

Denote a distribution in $\tilde {\cal Q}_n$ which is contaminated according to \eqref{contbed1} resp.\ \eqref{contbed2} by $Q_{n}^0$.
By the previous considerations, under $Q_{n}^0$, we may consider $|\tilde t|$ as being exactly $kb$,
and we will consider the cases $\tilde t=\pm kb$ simultaneously. For the substitution $t^\natural=\pm kb/\sqrt{\bar n}$, the
following abbreviations are convenient
\begin{equation} \label{knatdef}
\tilde k:=k/\sqrt{n},\qquad k^\natural:=k/\sqrt{\bar n}=\tilde k/n^\natural
\end{equation}
Taking up the dependency on $t^\natural$ in $h_n(s)$ as $h_n(s)=h_n(s,t^\natural)$,
in the {\tt MAPLE} procedure {\tt ash}, we introduce
\begin{equation}
\tilde h_n(s)=\tilde h_n(s,k^\natural)= h_n(s,k^\natural b)
\end{equation}
%
\paragraph{\small Integration w.r.t.\ $s$} \label{intssec}
In this step we integrate out $s$ in $\tilde h_n(s)$.
As $bl_n/v_0>\sqrt{2\log(n)}$, by Lemma~\ref{lemnormlog}, we may drop the integration limits and get
\begin{equation}
n\,{\rm MSE}(S_n,Q_n^0\,\big|K=k\,)=(n^\natural)^{-2}\int_{- \infty}^{ \infty} \, \tilde h_n(s)\varphi(s)\,\lambda(ds)
+\Lo(n^{-1})\label{intds2}
\end{equation}
So for integration, we use that for $X\sim{\cal N}(0,1)$, $\Ew[X^j]=0$, for ${j=1,3,5,7}$, and
\begin{equation}
\Ew[X^2]=1 ,\quad\Ew[X^4]=3 ,\quad\Ew[X^6]=15 ,\quad\Ew[X^8]=115
\end{equation}
and get (by our {\tt MAPLE} procedures {\tt intesout} and {\tt asMSEK})
\begin{eqnarray}
n\,{\rm MSE}(S_n,Q^0_n\,\big|K=k\,)&=&
\Lo(n^{-1})+(n^\natural)^{-2}\Big[(k^\natural)^2 b^2+v_0^2+
\Tfrac{1}{\sqrt{\bar n}}[\pm (3l_2+4\tilde v_1)v_0^2 k^\natural b \pm l_2(k^\natural)^3b^3]+\nonumber\\
&& + \Tfrac{1}{\bar n}\big[(\Tfrac{5}{4} l_2^2 +\Tfrac{1}{3}l_3)(k^\natural)^4 b^4 +(3 \tilde v_2 + 2 l_3 + 3 \tilde v_1^2  + \Tfrac{15}{2} l_2
 + 12 \tilde v_1 l_2)v_0^2 (k^\natural )^2 b^2+\nonumber\\
 && +(\rho_0 (2\tilde v_1+l_2) + \Tfrac{2}{3} \rho_1) v_0^3)+(12 \tilde v_1 l_2 + l_3 + 3 \tilde v_2 +
 \Tfrac{15}{4} l_2^2  + 9 \tilde v_1^2 ) v_0^4\big]\Big]\label{intsout}
\end{eqnarray}
As mentioned in Remark~\ref{BigRem}(c), the terms of $\kappa_0$ cancel out for $A_2$
as do the terms of $\rho_0$ for $A_1$.
\paragraph{\small Collection of terms}
As we want to calculate the expectation with respect to $K$,
we have to expand terms in a way that $k$ is only appearing in integer powers and in the nominator.
For this purpose we employ our {\tt MAPLE} procedures {\tt asNn}, {\tt asKn}, and get
\begin{equation}
(n^{\natural})^{-2}=1+\Tfrac{\tilde k}{\sqrt{n}}+\Tfrac{\tilde k^2}{n}+\Lo(n^{-1})\label{mnatural},\qquad
(n^{\natural})^{-3}=1+\Tfrac{3\tilde k}{2\sqrt{n}}+\Lo(n^{-1/2}),\qquad (n^{\natural})^{-4}=1+\Lo(n^{0})
\end{equation}
\begin{equation}
k^\natural=\tilde k+\Tfrac{\tilde k^2}{2\sqrt{n}}+\Lo(n^{-1/2}), \;\; 
(k^\natural)^2
=\tilde k^2+ \Tfrac{\tilde k^3}{\sqrt{n}}+\Tfrac{\tilde k^4}{n}+\Lo(n^{-1}),\;\;
(k^\natural)^3
=\tilde k^3+\Tfrac{3 \tilde k^4}{\sqrt{n}}+\Lo(n^{-1/2}),\;\; 
(k^\natural)^4
=\tilde k^4+\Lo(n^{0})
\end{equation}
Substituting $k^\natural$ and $n^{\natural}$ by means of these expressions, we obtain
({\tt MAPLE} procedure {\tt asMSEk})
\begin{eqnarray}
&\hspace{-2em}&n\,{\rm MSE}(S_n,Q^0_n\,\big|\,K=k\,)=
\Lo(n^{-1})+{\tilde k}^{2}{b}^{2}+{v_0}^{2}+{\frac {[\pm \left(4\,\tilde v_1\,+
3\,l_2\,\right )b+1]\tilde k {v_0}^{2}+(2\pm l_2b)\,{\tilde k}^{3}{b}^{2}}{\sqrt {n}}}+\nonumber\\
&\hspace{-2em}&+
{\frac {\left (3\,{b}^{2}\pm 3\,l_2\,{b}^{3}+\left (\frac{5}{4}\,{l_2^2}+\frac{1}{3}\,l_3\right ){b}^{4}\right ){\tilde k}^{4}
+\left(3\tilde v_2+9{\tilde v_1}^{2}+\frac{15}{4}\,{l_2^2}+l_3+12\,l_2\,\tilde v_1\right){v_0}^{4}+
\left((l_2+2\tilde v_1)\,{\rho_0}+\frac{2}{3}\,{\rho_1}\right){v_0}^{3}}{n}}+\nonumber\\
&\hspace{-2em}&+{\frac {
\left (\left (3\,{\tilde v_1}^{2}+3\,\tilde v_2+
12\,l_2\tilde v_1+\frac{15}{2}\,{l_2^2}+2\,l_3\right )\,{b}^{2}+1
\pm \left (6\,l_2\,+8\,\tilde v_1\,\right )b\right ){\tilde k}^{2}v_0^2}{n}}
\end{eqnarray}
\paragraph{\small Integration w.r.t.\ $\tilde k$}
As by Corollary~\ref{corewk} the event $\{K>(1+\delta)r\sqrt{n}\}$ only attributes $\Lo(n^{-1})$ to the
expectation of $\Ew[K^{j}]$, $j=0,\ldots,4$, we can now simply use Lemma~\ref{binlem} to determine
the MSE. This gives the result by our {\tt MAPLE} procedures {\tt intekout}, {\tt asMSE}:
\begin{equation}
  n\,\Ew_{Q_n^0}[S_n^2\,]=
  {r}^{2}{b}^{2}+{v_0}^{2}+\Tfrac{r}{\sqrt{n}\,}\,A_1+\Tfrac{1}{n}\,A_2+\Lo(n^{-1}) \label{8.70}
\end{equation}
with
\begin{eqnarray}
A_1&=&  {v_0}^{2}\,\Big( \pm(4\,\tilde v_1+3\,l_2 \,)b+1 \Big)+{b}^{2} +
 [2\,{b}^{2}\pm l_2\,{b}^{3} \,]\,{r}^{2} \label{8.71}\\
A_2&=&{{v_0}^{3}\,\Big((l_2+2\,\tilde v_1 \,)\rho_0+\Tfrac{2}{3}\,\rho_1\Big)+
 {v_0}^{4}\,(3\,\tilde v_2+{\Tfrac {15}{4}}\,{l_2^2}+l_3+9\,{\tilde v_1}^{2}+
 12\,\tilde v_1\,l_2 \,)}+\nonumber\\
 &&\quad+\big[\, {v_0}^{2}\,\Big( (3\,\tilde v_2+3\,{\tilde v_1}^{2}+\Tfrac{15}{2}\,{l_2^2}+2\,l_3+
 12\,\tilde v_1\,l_2 \,){b}^{2}+1\pm (8\,\tilde v_1+6\,l_2 \,)\,b \Big)\pm 3\,l_2\,{b}^{3}+
 5 \,{b}^{2} \,\big]\,{r}^{2}+\nonumber\\
 &&\qquad+\Big( (\Tfrac{5}{4}\,{l_2^2}+\Tfrac{1}{3}\,l_3 \,){b}^{4}
 \pm 3\,l_2\,{b}^{3} +3\,{b}^{2}\Big)\,{r}^{4} \label{8.72}
\end{eqnarray}
\paragraph{\small Decision upon the alternative \eqref{contbed1} or \eqref{contbed2}}\label{Decalter}
Denote $Q_n^-$ a contaminated member in $\tilde {\cal Q}_n(r)$ according to \eqref{contbed1} and
correspondingly $Q_n^+$ according to \eqref{contbed2}. With respect to terms of \eqref{8.70}--\eqref{8.72},
obviously, if $\sup\psi < -\inf \psi$, the maximal MSE is achieved by $Q_n^-$, respectively by $Q_n^+$ if $\sup\psi > -\inf \psi$.
In case $\sup \psi=-\inf\psi$, the terms in $A_1$ are decisive:
\begin{equation}
n(\Ew_{Q_n^+}[S_n^2\,]-\Ew_{Q_n^-}[S_n^2\,])=\Tfrac{rb}{\sqrt{n}}\Big\{l_2\big[(r^2b^2+3v_0^2)(1+2\Tfrac{r}{\sqrt{n}})
 +\Tfrac{3b^2r(r^2+1)}{\sqrt{n}}\big]+
4v_0^2(1+\Tfrac{r}{\sqrt{n}}) v_1\Big\}+\Lo(n^{-1})
\end{equation}
Hence, $Q_n^-$ [$Q_n^+$] is least favorable up to $\Lo(n^{-1})$ if
\begin{equation} \label{wokontlr}
\tilde v_1\,>\,[<]-\Tfrac{l_2}{4}\Big(\Tfrac{b^2}{v_0^2}(r^2+3)(1+\Tfrac{r}{\sqrt{n}}-\Tfrac{2r^2}{n})+3(1-\Tfrac{b^2}{v_0^2})\Big)
\end{equation}
If there is ``$=$'' in \eqref{wokontlr}, no decision can be taken up to order $\Lo(n^{-1})$.
\hfill\qed

\subsection{Proofs to Propositions~\ref{abklingP} and \ref{slowdecay}}\label{abklingp}
For $\ve_1 \in (0,1)$, let $N_{\SSs +}(t)=N_{\SSs +}(t;n,\ve_1,\hat b)$, $N_{\SSs -}(t)=N_{\SSs -}(t;n,\ve_1,\check b)$ be defined as
\begin{equation}
  N_{\SSs +}(t):=\#\Big\{  \psi(x_i-t)\ge \hat b (1-\ve_1),\;U_i=0\Big\},\qquad
  N_{\SSs -}(t):=\#\Big\{  \psi(x_i-t)\le \check b (1-\ve_1),\;U_i=0\Big\}
\end{equation}
The idea behind Propositions~\ref{abklingP} and \ref{slowdecay} is to use the inclusions
\begin{eqnarray}
\big\{\Tsum \psi(x_i-t) \leq 0 \big\} \subset \big\{N_{\SSs +}(t) \le n_{\SSs +} \big\},\;\;  
\big\{\Tsum \psi(x_i-t) \geq 0 \big\} \subset
\big\{N_{\SSs -}(t) \le n_{\SSs -} \big\} \label{N-inkla}
\end{eqnarray}
for some numbers $n_{\SSs -}$, $n_{\SSs +}$  yet to be specified.\\
For Proposition~\ref{abklingP}, symbolically in the tableau of page~\pageref{tabelAUF}, we plug in $\delta=0$,
so that the second and third line are separated by $|t|=Cn$. All cases except for case (IV) remain unchanged.
For (IV), we consider the first inclusion of \eqref{N-inkla}.  In this case,
 $\{\Tsum \psi(x_i-t) \leq 0 \}$ is distorted most importantly by $\tilde t= k\hat b$.
On the other hand the $N''=n-N_{\SSs +}-K$ remaining observations cannot be smaller than $N''\check b$, so
\begin{equation}
\Tsum \psi(x_i-t) \leq 0\qquad \Longrightarrow\qquad N_{\SSs +}\hat b(1-\ve_1)+ K \hat b +N'' \check b \leq 0
\end{equation}
that is
$
N_{\SSs +} \leq \Big(-n \check b - K(\hat b -\check b)\Big)\bigg/\Big(\hat b (1-\ve_1) -\check b\Big)
$, 
and as this has to hold for all $K\leq \ve_0' n$,
$
N_{\SSs +} \leq  n \Big(- \check b - \ve_0'(\hat b -\check b)\Big)\bigg/\Big(\hat b (1-\ve_1) -\check b\Big)=:n_{\SSs +}=n_{\SSs +}(\ve_0')
$,  
where by \eqref{s0def} and as $0<\ve_1<1$, we get $n_{\SSs +}=n \ve_{\SSs +}$ for
\begin{equation}\label{posrel}
0< \ve_{\SSs +}=\Big(- \check b - \ve_0'(\hat b -\check b)\Big)\bigg/\Big(\hat b (1-\ve_1) -\check b\Big)<1-\ve_0'
\end{equation}
Accordingly, for the second inclusion in \eqref{N-inkla}, we obtain
\begin{equation}
N_{\SSs -}\leq n\ve_{\SSs -}=:n_{\SSs -}=n_{\SSs -}(\ve_0') \qquad \mbox{for}\quad \ve_{\SSs -}:=\Big(\hat b - \ve_0'(\hat b -\check b)\Big)\bigg/\Big(\hat b - \check b(1-\ve_1) \Big)
\end{equation}
where again $0<\ve_{\SSs -}<1-\ve_0'$. Hence with $\bar k =\ulcorner \ve_0' n\urcorner-1$
\begin{eqnarray}
\hspace{-2em}\Pr\{S_n> \sqrt{t} \,\Big|\,D_{k,\tilde t=-k\check b}\}
&\stackrel{\SSs\eqref{mest2}}{\leq}&
\Pr\Big\{T_{n,k}(\sqrt{t}\,)\ge k\check b\Big\}\leq
\Pr \Big\{T_{n,k}(\sqrt{t}\,)\ge \bar k \check b\Big\}\leq 
\Pr\Big\{N_{\SSs -}(\sqrt{t\,}\,) \leq n_{\SSs -}\,\big|\, K= \bar k\Big\} \label{binom1}
\end{eqnarray}
and correspondingly
$\Pr\Big\{S_n< -\sqrt{t} \,\Big|\,D_{k,\tilde t=k\hat b}\Big\}
\leq
\Pr\Big\{N_{\SSs +}(-\sqrt{t\,}\,) \leq n_{\SSs +}\big|\, K= \bar k\Big\}
$. 
But,  ${\cal L}(N_{\SSs \pm}|K=k)$ is ${\rm Bin}(n-k,p_{\SSs \pm})$ for
\begin{equation}
p_{\SSs -}(t)=\Pr\big(\psi(X^{\SSs \rm id}-\sqrt{t\,}\,)\leq \check b (1-\ve_1)\big)=,\qquad
p_{\SSs +}(t)=
\Pr\big(\psi(X^{\SSs \rm id}+\sqrt{t\,}\,)\geq \hat b (1-\ve_1)\big)
\end{equation}
That is, $p_{\SSs -}(t)=F(\sqrt{t\,}\,+ B_{\SSs -})$, $p_{\SSs +}(t)=\bar F(\sqrt{t\,}\,- B_{\SSs +})$
where $\bar F=1-F$ and
\begin{equation}
B_{\SSs -}:=\inf\Big\{y\,\big|\,\psi(y)\ge (1-\ve_1) \check b \Big\},
\qquad B_{\SSs +}:=\sup\Big\{y\,\big|\,\psi(y)\le (1-\ve_1) \hat b \Big\}
\end{equation}
If we abbreviate $m=n-\bar k$, $m_{\SSs \pm} =\ulcorner n_{\SSs \pm} \urcorner$,
$p_t=(1-p_{\SSs +}(t))\vee p_{\SSs -}(t)$,
in the binomial probabilities in \eqref{binom1}, we obtain ${{m}\choose{j}} \leq 2^n$,
 $j=0,\ldots m_{\SSs \pm}$, and $p_{\SSs -}(t), (1-p_{\SSs +}(t))\leq 1$,
so that
\begin{eqnarray}
&&\sup_k \Pr\Big\{|S_n|> \sqrt{t} \,\Big|\,D_{k,|\tilde t|=k\check b}\Big\}
\leq n 2^n p_t^{[m-(m_{\SSs -}\vee m_{\SSs +})]}
\end{eqnarray}
But by \eqref{posrel}, $1-\ve_0'-(\ve_{\SSs -}\vee \ve_{\SSs +})=:\alpha >0$, so
$
m-(m_{\SSs -}\vee m'_{\SSs +})\ge \alpha n-1
$. 
Now, by \eqref{abkling}, for $\hat B=\max\{B_{\SSs +},-B_{\SSs -}\}$, if $n$ is so large that $C n>(T-\hat B)^2$,
\begin{eqnarray*}
&&\sup_k \int_{Cn}^\infty \Pr\Big\{|S_n|> \sqrt{t\,}\, \,\Big|\,D_{k,|\tilde t|=kb}\Big\}\leq
 n 2^{n+1} \int_{Cn}^\infty  t^{-\eta(\alpha n-1)/2}\,dt=
 \exp[-\tilde\alpha n \log(n)(1-\Lo(n^0))]
\end{eqnarray*}
for some $\tilde\alpha'>0$. So (IV) is indeed negligible.\hfill\qed

For Proposition~\ref{slowdecay}, we only show the first case of \eqref{abkl};
the second follows analogously. This time $K=0$, $n$ is fixed, and we use the inclusions of the
complements in \eqref{N-inkla}. Thus
\begin{eqnarray*}
&&\Pr\{S_n\ge \sqrt{t}\}
\stackrel{}{\ge}
\Pr\big\{T_{n,0}(\sqrt{t}\,)> 0\big\}\ge \Pr\big\{N_{\SSs +}(\sqrt{t}) >  n_{\SSs +}(0)\big\}
\end{eqnarray*}
Let
$\tilde p_{\SSs +}=\bar F(\sqrt{t}+B_{\SSs +})$.
To $\delta>0$ there is a $T>0$ such that for $t>T$ and $\tilde p_{\SSs +}^{n_{\SSs +}}>1-\delta$.
Hence for 
$t>T^2$ and $n'=m_{\SSs +}+1$
\begin{eqnarray*}
&&\Pr\{S_n> \sqrt{t\,}\, \}
\geq {{n}\choose{n'}} (1-\tilde p_{\SSs +})^{n'} \tilde p_{\SSs +}^{n-n'}
\geq {{n}\choose{n'}} (1-\delta) \bar F(\sqrt{t\,}\,+B_{\SSs +})^{n'}
\end{eqnarray*}
Now by the first half of \eqref{abkl}, for $d=1/n'$ and some $c>0$, $T'>T$ and for all $t>T'$
\begin{equation}
t^{1/n'}\big(1-F(t)\big) > c
\qquad\iff\qquad  \big(1-F(t)\big)^{n'} > c^{n'} t^{-1}
\end{equation}
Then for the M-estimator $S_n$,
\begin{eqnarray}
\Ew_F[({\rm S}_n)_{\SSs +}^2]
&\geq& \int_{(T')^2}^{\infty}\Pr\Big\{S_n> \sqrt{t\,}\, \Big\}\,dt
\geq \int_{(T')^2}^{\infty}
{{n}\choose{n'}} (1-\delta) c^{n'} (\sqrt{t\,}\,+B_{\SSs +})^{-1}\,dt=\infty
\end{eqnarray}
\hfill\qed
\subsection{Proof of Proposition~\ref{neccond}}\label{neccondp}
For $t>v_0^2 \log(n)/n$, we consider the following inclusion
$$
\Big\{\psi(x-\sqrt{t\,}\,)>b-c_0/\sqrt{n}\Big\}=\Big\{x>\sqrt{t\,}+B_n\Big\}
  \subset \Big\{x>v_0\sqrt{\log(n)/n}+B_n\Big\}$$
Let
$ 
A_{k,t}:=\Big\{\sum_{i\colon U_i=1}\psi(X_i-\sqrt{t\,}\,) \leq (k-1) (b-c_0/\sqrt{n\,}\,)\Big\}
$.  
Hence if $t>v_0^2 \log(n)/n$, by \eqref{condnec},
for all $k>(1-\delta)r\sqrt{n}$,
\begin{equation}\label{restri}
\Pr(A_{k,t}\,\Big|\,K=k)\ge p_0
\end{equation}
Now we proceed as in section~\ref{prmain}, and even with restriction~\eqref{restri} the arguments of
subsection~\ref{lfselec} remain in force, so that we have to maximize $t^\natural$.
But $t>v_0^2 \log(n)/n \iff s>\sqrt{\log n }$ in \eqref{letz3eq}. Hence on the event $A_{k,t}$
for $s\in [\sqrt{\log n };bl_n/v_0)$, we get the bound  $t^\natural\leq (k^\natural-1) (b-c_0/\sqrt{n})/\sqrt{\bar n}$,
while for $s\in (-bl_n/v_0;\sqrt{\log n })$ respectively on ${}^cA_{k,t}$, we  bound $t^\natural$ by $k^\natural b$.
Integrating out these two $s$-domains separately as in subsection~\ref{intssec}, we obtain for $\Delta_n=n\,\Big({\rm MSE}(S_n,Q^0_n\,\big|K=k\,)-{\rm MSE}(S_n,Q^\flat_n\,\big|K=k\,)\Big)$
\begin{eqnarray*}
&&  \Delta_n \geq 
p_0 \int_{\sqrt{\log n }}^{bl_n/v_0}\Big(2 v_0s D_n(\tilde k)+2\tilde kb D_n(\tilde k)-D_n(\tilde k)^2\Big)\,\varphi(s)\,ds+\Lo(n^{-1})
\end{eqnarray*}
for
$ 
D_n(\tilde k)=\tilde k c_0/\sqrt{n}+b/\sqrt{n}+\Lo(1/\sqrt{n})
$. 
But for $0<a_1<a_2<\infty$,
$
\varphi(a_1)/a_2-\varphi(a_2)/a_2\le \int_{a_1}^{a_2} \,\varphi(s)\,ds
$,
so that with $a_1=\sqrt{\log n }$, $a_2=bl_n/v_0$, and as $\varphi(a_2)=\Lo(n^{-1})$,
\begin{eqnarray*}
&&  \Delta_n \geq 
\frac{p_0}{\sqrt{2\pi n}}[2 v_0  D_n(\tilde k)-2 \Tfrac{\tilde kb D_n(\tilde k)+D_n(\tilde k)^2}{bl_n/v_0}]+\Lo(n^{-1})=
\frac{2p_0v_0}{\sqrt{2\pi n}}D_n(\tilde k)+\Lo(n^{-1})
\end{eqnarray*}
Now the restriction to $(1-\delta)r\sqrt{n}<K<k_1r\sqrt{n}$ by Lemma~\ref{binlem} may be dropped, giving
$
\Delta_n \geq\frac{2p_0v_0}{n\sqrt{2\pi}}(rc_0+b)+\Lo(n^{-1})\nonumber
$.
\hfill\qed

\section{Auxiliary Results}\label{appsec}\req

\subsection{Two Hoeffding Bounds}
%
\begin{Lem}\label{hoef}
  Let $\xi_i \iid F$,  $i=1,\ldots,n$ be real--valued random variables, $|\xi_i|\leq M$
  Then for $\ve>0$
  \begin{eqnarray}
    P(\frac{1}{n}\sum_i \xi_i -\Ew[\xi_1] \geq \ve) &\leq& \exp(-\frac{2n\ve^2}{M^2}),\qquad 
    P(\frac{1}{n}\sum_i \xi_i -\Ew[\xi_1] \leq -\ve) \leq \exp(-\frac{2n\ve^2}{M^2})\label{hoe2}
  \end{eqnarray}
\end{Lem}
\begin{proof}{}
  \citet[Thm.~2. and Thm.~1, inequality (2.1)]{Hoef:63}.
\hfill\qed\end{proof}
\subsection{A uniform Edgeworth expansion}
%
%
In the following theorem, generalizes \citet[Thm.~1]{Ibr:67} and \citet[Thm.~3.3.1]{Ib:Li:71} to the situation where the
law of $\xi_i$ depends through an additional parameter $t$:
\begin{Thm}\label{berry2}
  For some set $\Theta \subset \R $ and fixed $t\in \Theta$ let $\xi_{i,t}$, $i=1,2,\ldots$ be a sequence of i.i.d.\ real-valued
  random variables with distribution $F_t$ and with
  \begin{equation}
\Ew \xi_{i,t}=0,\quad  \Ew \xi_{i,t}^2=1, \quad  \Ew \xi_{i,t}^3=\rho_t, \quad  \Ew \xi_{i,t}^4-3=\kappa_t
  \end{equation}
Let $\Phi(s)$ and $\varphi(s)$ be the c.d.f.\ and p.d.f.\ of ${\cal N}(0,1)$ and
\begin{eqnarray}
  F_n(s,t)&:=&P(\Tsum_{i=1}^n \xi_{i,t} < s\,\sqrt{n}),\qquad 
  H_n(s,t):=\Phi(s)-\frac{\varphi(s)}{\sqrt{n}\,}\varphi(s)\frac{\rho_t}{6}(s^2-1)\\
  G_n(s,t)&:=&H_n(s,t)-\frac{\varphi(s)}{n}\Big[\frac{\kappa_t}{24}(s^3-3s)
  +\frac{\rho_t^2}{72}(s^5-10s^3+15s)\Big]\label{Gnstdef}
\end{eqnarray}
Let $f_t$ be the characteristic function  of $F_t$.
\begin{ABC}
\item If $\sup_{t}\kappa_t<\infty$ and if
 there is some $u_0>0$ such that for all $u_1$ the ``no-lattice''-condition
 (C)'
\begin{equation}
\hat f_{u_0}(u_1):=\sup_{u_0<u<u_1} \sup_{t}|f_t(u)|<1  \label{Cddd2}
\end{equation}
is fulfilled, then
\begin{equation}
\sup_{s\in\R} \sup_{t} |F_n(s,t)-H_n(s,t)|=\Lo(n^{-1/2})
\end{equation}
\item If
$
\sup_{t}\Ew |\xi_{i,t}|^5<\infty
$ and 
the uniform Cram\'er--condition (C)
\begin{equation}
\limsup_{u\to\infty} \sup_{t}|f_t(u)|<1  \label{Cddd}
\end{equation}
is fulfilled, then  
\begin{equation}
\sup_{s\in\R} \sup_{t} |F_n(s,t)-G_n(s,t)|=\LO(n^{-3/2})
\end{equation}
\end{ABC}
\end{Thm}
\begin{proof}{}
The general technique to prove Edgeworth expansions is to use Berry's smoothing lemma, which we take from \citet[Thm.~1.5.2]{Ib:Li:71}
and apply it to our case: Let $f_{n,t}$ be the characteristic function of $F_{n}(\,\cdot\,,t)$, and
define the Edgeworth measures $G_{n,j,t}$, $j=1,2$ as
$
G_{n,1,t}(s)=H_n(s,t)$, $G_{n,2,t}(s)=G_n(s,t)
$
as well as their Fourier-Stieltjes transforms
$g_{n,j,t}(u)=\int e^{isu} G_{n,j,t}'(s)\,\lambda(ds)$
and
$
\hat G_{n,j}'=\sup_t \sup_{s\in\R} |G_{n,j,t}'(s)|
$. 
Then for $T>T'>0$, it holds that
\begin{eqnarray}
\sup_{s\in\R} \sup_t |F_{n}(s,t)-G_{n,j,t}(s)|&\leq&
\sup_t \frac{1}{\pi} \int_{-T'}^{T'} \frac{|f_{n,t}(u)-g_{n,j,t}(u)|}{|u|}\,\lambda(du)+
\sup_t \frac{1}{\pi} \int_{T'\leq |u|<T} \frac{|f_{n,t}(u)|}{|u|}\,\lambda(du)+\nonumber\\
&&\quad+
\sup_t \frac{1}{\pi} \int_{T'\leq |u|<T} \frac{|g_{n,j,t}(u)|}{|u|}\,\lambda(du)+
\sup_t \frac{24}{\pi T}\hat G_{n,j}' \label{Berrybound}
\end{eqnarray}
But similarly as in \citet[pp.~462]{Ibr:67}, for some constants $\gamma>0$  and $c_j>0$, we get for $T'=\gamma \sqrt{n}$
and $|u|\leq T'$
\begin{equation}
\frac{|f_{n,t}(u)-g_{n,j,t}(u)|}{|u|} \leq c_j \sup_t \Ew[|\xi_{1,t}|^{3+j}]\,  n^{-(j+1)/2}\,(|u|^{j}+|u|^{2+3j})\,e^{-u^2/4}
\end{equation}
and hence the first summand in the RHS of \eqref{Berrybound} is $\LO(n^{-(j+1)/2})$.
For the second summand, we note that $f_{n,t}(u)=f_{t}^n(u/\sqrt{n})$ and hence
\begin{equation}
\int_{T'}^{T} \frac{|f_{n,t}(u)|}{u}\,\lambda(du)=\int_{\gamma}^{T/\sqrt{n}} \frac{|f_t^n(u)|}{u}\,\lambda(du)
\end{equation}
In case $j=2$, for $\gamma$ sufficiently large, by condition~(C), $\sup_t \sup_{|u|>\gamma} |f_t(u)|=:\beta<1$ and hence,
for $T=n^{3/2}$,
\begin{equation}
\sup_t \int_{T'}^{T} \frac{|f_{n,t}(u)|}{u}\,\lambda(du)\leq \log(T/\sqrt{n\,}\,) \beta^n=\Lo(e^{-\sqrt{n}/2})
\end{equation}
In case $j=1$, we proceed as in \citet[Lemma~3.3.1]{Ib:Li:71}: If $\sup_{u_1} \hat f_{\gamma}(u_1) < 1$ for $\gamma$ sufficiently large, we may proceed as in case $j=2$;
else,  (C') says that for  $\gamma$ sufficiently large, $\hat f_{\gamma}(u_1)$ is isotone in $u_1$ and tends to $1$. So we may define
\begin{equation}
l'_n:=\inf\{u_1\,\big|\,\hat f_{\gamma}(u_1)\ge 1-1/\sqrt{n\,}\,\}
\end{equation}
Setting $T=\sqrt{n\,}\,l_n$ for $l_n=\min(l'_n,\sqrt{n})$, we see that $l_n^{-1}=\Lo(n^{0})$ and
\begin{equation}
\sup_t \int_{T'}^{T} \frac{|f_{n,t}(u)|}{u}\,\lambda(du)\leq \log(\sqrt{n}) (1-1/\sqrt{n})^n\leq
\log(\sqrt{n}) e^{-\sqrt{n}}=\Lo(e^{-\sqrt{n}/2})
\end{equation}
Hence the second summand in the the RHS of \eqref{Berrybound} is $\LO(n^{-(j+1)/2})$.
Also, it is easy to see that $\hat G_{n,j}'<\infty$, and hence by the choice of $T$, the last summand in the the RHS of \eqref{Berrybound}
is $\LO(l_n^{-1}n^{-1/2})=\Lo(n^{-1/2})$ in case $j=1$, and $\LO(n^{-3/2})$ for $j=2$.
Finally, by  Mill's ratio, the third summand is again easily shown to be $\LO(\exp(-\gamma^2 n/3))$.
\hfill\qed\end{proof}
\subsection{Moments for the Binomial}
\begin{Lem}\label{binlemmom}
  Let $X\sim {\rm Bin}(n,p)$. Then
  \begin{align}
    &\Ew[X]=pn,\qquad\Ew[X^2]=p^2n^2+pn-p^2n,\\
    &\Ew[X^3]=p^3n^3-3p^3n^2+2p^3n+3p^2n^2-3p^2n+pn,\\
    &\Ew[X^4]=p^4n^4-6p^4n^3+11p^4n^2-6p^4n+
    6p^3n^3-18p^3n^2+12p^3n+7p^2n^2-7p^2n+pn
  \end{align}
and consequentially, for $p=r/\sqrt{n}$,
  \begin{align}
    &\Ew[X]=rn^{1/2},\qquad\Ew[X^2]=r^2n+rn^{1/2}-r^2,\\
    &\Ew[X^3]=r^3n^{3/2}+3r^2n+(r-3r^3)n^{1/2}-3r^2+2r^3n^{-1/2},\\
    &\Ew[X^4]=r^4n^2+6r^3n^{3/2}+(7r^2-6r^4)n+(r-18r^3)n^{1/2}+
    11r^4-7r^2+12r^3n^{-1/2}-6r^4n^{-1}
  \end{align}
\end{Lem}
\begin{proof}{}
  easy calculations for {\tt MAPLE} --- see procedure {\tt Binmoment}\ldots
\hfill\qed\end{proof}
\subsection{Decay of the standard normal}
Finally, we note  the following Lemma for ${\cal N}(0,1)$ variables
\begin{Lem} \label{lemnormlog}
  Let $X\sim {\cal N}(0,1)$. Then for $0\leq k\leq 8$ and any sequence $(c_n)_n\subset\R$ with $\liminf_n c_n>\sqrt{2}$,
  \begin{equation}
    \Ew [X^k\Jc_{\{X\geq c_n\sqrt{\log(n)}\}}] =\Lo(n^{-1})
  \end{equation}
\end{Lem}
\begin{proof}{}\citet[Lem.~A.6]{Ruck:03b}.
\hfill\qed\end{proof}
%
\section*{Acknowledgement}
\medskip

\medskip
\hrulefill\hspace*{6cm}\\[2ex]
Web-page to this article:\\
\href{http://www.mathematik.uni-kl.de/~ruckdesc/}%
{{\footnotesize \url{http://www.mathematik.uni-kl.de/~ruckdesc/}}}

%
%
%
%
%
\end{document}